\documentclass[a4paper,12pt,oneside]{amsart}
\usepackage[cm]{fullpage}
\usepackage{typearea}
\typearea{15}

\usepackage{mymacros}

\usepackage{mathtools}
\usepackage[colorlinks=false]{hyperref} 
\usepackage[all]{xy}
\usepackage{multirow}

\newcommand\SU{\operatorname{SU}}
\newcommand\Spin{\operatorname{Spin}}
\newcommand\Aff{\operatorname{Aff}}

\newcommand\AdS{\operatorname{AdS}}
\newcommand\dS{\operatorname{dS}}

\newcommand\pr{\operatorname{pr}}
\newcommand\ad{\operatorname{ad}}
\newcommand\nablabar{\overline\nabla}
\renewcommand{\labelenumi}{(\arabic{enumi})}

\title{Generalized Killing spinors as Lagrangian sections of principal spin bundles}
\author{Fumihiro Ueno}
\email{ueno-fumihiro-dx@alumni.osaka-u.ac.jp}

\begin{document}
	\begin{abstract}
		In this paper, we study a one-to-one correspondence between $\epsilon$-unit generalized Killing spinors on a three-dimensional pseudo-Riemannian spin manifold $(M,g)$ and Lagrangian sections of an almost pseudo-Hermitian structure on its principal spin bundle.
	\end{abstract}
	\maketitle
	\section{Introduction}
	A generalized Killing spinor $\psi \in \Gamma(\Sigma M) $ on a pseudo-Riemannian spin manifold $ (M,g) $ is a spinor field satisfying the following differential equation for every $\scX\in \Gamma(TM)$,
	\begin{equation*}
		\nabla_{\scX}\psi = \scA(\scX)\cdot\psi, 
	\end{equation*}
	where $ \scA $ is a real or purely imaginary $g$-symmetric $ (1,1) $-type tensor field, $ \scA(\scX)\cdot \psi $ denotes Clifford product, and $\nabla$ is the spin Levi-Civita connection. 
	Generalized Killing spinors arise as restrictions of parallel spinors to hypersurfaces \cite{bar2005generalized}.
	In dimension three, Moroianu and Semmelmann \cite{moroianu2014generalized} establish a one-to-one correspondence, up to sign, between unit-length generalized Killing spinors and divergence-free orthonormal frames on Riemannian spin manifolds. They further show that, on the three-sphere, this correspondence identifies unit-length generalized Killing spinors with Lagrangian graphs in the nearly K\"ahler manifold $\bS^3\times\bS^3$.
	
	The aim of the present paper is to extend this correspondence from the three-dimensional sphere to arbitrary three-dimensional pseudo-Riemannian spin manifolds. In this extension, Lagrangian graphs in the nearly K\"ahler manifold $\bS^3\times\bS^3$ are replaced by Lagrangian sections of the almost pseudo-Hermitian structure on the principal spin bundle. 
	\begin{theorem}\label{th:GKS-LagGraph-main}
		For each $\epsilon\in\scE_r$, there is a
		one-to-one correspondence between $\epsilon$-unit generalized
		Killing spinors on a three-dimensional pseudo-Riemannian spin manifold $(M,g)$ and Lagrangian sections of almost pseudo-Hermitian structure $(\scG_c,\scJ_c)$ on the total space of its principal spin bundle $\pi\colon P\to M$.
		Thus, in the Lorentzian case, the correspondence is two-to-one if
		spacelike and timelike spinors are considered together.
	\end{theorem}
	Note that these Lagrangian sections are all non symplectic Lagrangian but almost symplectic Lagrangian as in \pref{rm:non-symplectic_Lagrangian}.
	Then it is natural to investigate when an almost complex structure becomes integrable and when an almost pseudo-Hermitian structure becomes nearly pseudo-K\"ahler structure defined on the principal spin bundle, both of which are obtained in \pref{th:integrability-complex} and \pref{th:NK-str} respectively.
	These conditions are so restrictive that constant sectional curvature is necessary; the former condition requires the sectional curvature to be $(-1)^{r-1}c^2$, whereas the latter requires it to be $\frac{(-1)^r}{3}c^2$.
	In a related direction, Zentner studied almost complex structures on the total spaces of principal bundles \cite{zentner2013integrable}.
	It would also be interesting to investigate the classification of almost pseudo-Hermitian structures on principal spin bundles.
	This question is naturally related to the Gray--Hervella classification of almost Hermitian structures into sixteen classes \cite{gray1980sixteen}, as well as to its pseudo-Hermitian counterpart discussed in \cite{brozos2012geometric}.

	The Lorentzian analogue of the three-dimensional sphere is the anti-de Sitter three-spacetime $(\AdS_3)$, a Lorentzian space form of constant negative sectional curvature.
	Following the method based on Hopf vector fields on $\bS^3$ \cite{moroianu2014generalized_Einstein, moroianu2014generalized}, we construct examples of $\epsilon$-unit generalized Killing spinors on $\AdS_3$ and the corresponding nearly pseudo-K\"ahler Lagrangian sections on its principal spin bundle, which is identified with $\AdS_3 \times \AdS_3$. These examples are summarized in \pref{tb:GKS-group-space-forms} and \ref{tb:H-valued-functions-group-space-forms}, together with the corresponding results for $\bS^3$ \cite{moroianu2014generalized}.
	In Lorentzian signature, a symmetric endomorphism associated with a generalized Killing spinor is not necessarily diagonalizable. We also construct a family of $\epsilon$-unit generalized Killing spinors with non-diagonalizable symmetric endomorphisms on $\AdS_3$.
	The corresponding family of Lagrangian submanifolds in the nearly pseudo-K\"ahler manifold $\AdS_3\times\AdS_3$ appears as one of the cases in the classification of extrinsically homogeneous Lagrangian submanifolds given in \cite{anarella2026extrinsically}. The construction of the associated generalized Killing spinors and Lagrangian sections is described in detail in \pref{sc:null-Jordan-calculations}. 
	
	For a Lie group $G$, Artacho classified left-invariant generalized Killing spinors in \cite[Theorem~4.2]{artacho2025generalised}. In our framework, invariant spinors correspond to $G$-equivariant sections of the principal spin bundle, which are represented by constant maps with respect to the natural left-invariant trivialization. In \pref{th:invariant-GKS-unimodular}, we show that the following conditions are equivalent: $G$ is unimodular, $G$ admits an invariant generalized Killing spinor of unit length, and the principal spin bundle admits a $G$-equivariant Lagrangian section. Moreover, whenever these conditions hold, every invariant spinor of unit length is a generalized Killing spinor and every $G$-equivariant section is Lagrangian. This recovers part of Artacho's result and extends it to Lorentzian signature. By contrast, in \pref{eg:Aff}, we construct non-invariant generalized Killing spinors on a non-unimodular Lie group of Bianchi type~III whose corresponding Lagrangian sections are represented by nonconstant maps.
	
	At the end of this paper, we construct examples of $\epsilon$-unit generalized Killing spinors on nonhomogeneous $\bR^3$ to generalize the left-invariant metric of $\operatorname{Sol} $ in \pref{eg:sol}.  
	In the final \pref{sc:nonhomogeneous}, we construct two types of Lagrangian sections. One of which corresponds to sections whose divergence-free frames are equal to the same frame appeared as left-invariant frame \eqref{eq:unimodular-invariant-frame} in \pref{eg:sol}. The other corresponds to non-trivial Lagrangian sections as generalization of non-invariant generalized Killing spinors on $\operatorname{Sol}$.
	
	\section{The Lagrangian correspondence}\label{sc:Lag-GKS-Correspondence}
	\subsection{Almost pseudo-Hermitian structure on the spin bundle}
	Let $(M=M^{r,3-r},g)$ be a pseudo-Riemannian spin manifold of dimension three, where $r=0$ corresponds to the Riemannian case and $r=1$ corresponds to the Lorentzian case.
	A spin structure on $(M,g)$ is a pair $(P,\Lambda)$ of a principal spin bundle $\pi\colon P\to M$ and a bundle homomorphism $\Lambda \colon P\to \SO(M,g) $ such that the following diagram commutes.
	\begin{equation*}\label{eq:spin-str}
		\def\objectstyle{\scriptstyle}
		\def\labelstyle{\scriptscriptstyle}
		\xymatrix{
			P\times \Spin_0(r,3-r)\ar[rr]^-{\Lambda \times \lambda}\ar[d]_-{R_\bullet(\bullet)}\ar@{}[rrd]|{\circlearrowright}&&\SO(M, g)\times \SO_0(\frakm)\ar[d]^-{R_\bullet(\bullet)}\\
			P\ar[rr]^-{\Lambda}\ar[dr]_-{\pi}&&\SO(M, g)\ar[dl]^-{\pr}\\
			&M\ar@{}[u]|{\circlearrowright}&
		}
	\end{equation*}
	In dimension three, we write \begin{equation*}
		(H,\frakh) = \begin{cases} (\SU(2), \mathfrak{su}(2)),& (r=0)\\ (\SU(1,1), \mathfrak{su}(1,1)),& (r=1) \end{cases}, \quad \frakm=(\bR^{r,3-r}, \la\bullet,\bullet\ra),
	\end{equation*} so that $H\simeq\Spin_0(r,3-r)$. We denote by \begin{equation*}
		\lambda\colon H\longrightarrow \SO_0(\frakm),\quad \lambda_*\colon \frakh\to \mathfrak{so}(\frakm)
	\end{equation*} the corresponding double covering homomorphism and its Lie algebra isomorphism.
	There is a canonical bundle isomorphism $TM\simeq P\times_\lambda \frakm.$
	In particular, each $ u \in P $ induces a linear isometry 
	\begin{equation*}
		\iota_u\colon \frakm\ni X\mapsto [u,X] \in T_{\pi(u)}M. 
	\end{equation*}

	Let $ s \in \Gamma(\scU, P) $ be a section and $\scX =[s,X] \in \Gamma(\scU;TM) $ for each $ X\in \frakm $ be a vector field on an open set $\scU\subset M$.
	Let $ (\theta, \omegatilde) $ be a pair of the solder form $\theta \in \Omega^1(P;\frakm) $ and the spin Levi-Civita connection form $\omegatilde \in \Omega^1(P;\frakh). $
	The corresponding Ehresmann connection is the following direct sum decomposition
	\begin{equation*}
		TP = \ker \omegatilde \oplus \ker d\pi.
	\end{equation*}  
	Since $ \ker d\pi = \ker \theta $, for every $ \scrX \in \Gamma(TP) $ and $u\in P, $ there exist unique $X\in \frakm $ and $ U\in \frakh $ such that
	\begin{equation*}
		\scrX|_u = X^H_u + U^\#_u,
	\end{equation*}
	where $X^H $ is a standard horizontal vector field satisfying $\theta(X^H) = X, $ and $U^\# $ is a fundamental vector field satisfying $ \omegatilde(U^\#)= U.$
	Thus, there exists a unique linear isomorphism
	\[ T_uP \ni \scrX_u \mapsto (\theta_u(\scrX_u),\omegatilde_u(\scrX_u)) \in \frakm \oplus \frakh, \quad \forall u\in P. \]
	
	Recall that an almost pseudo-Hermitian structure on an even dimensional manifold $M$ is a pair $(g,\scJ)$ of pseudo-Riemannian structure $g$ and an almost complex structure $\scJ$ on $M$ with
	\begin{equation*}
		g(\scJ \scX,\scJ \scY) = g(\scX,\scY), \quad \forall \scX,\scY\in \Gamma(TM).
	\end{equation*}
	We construct an almost pseudo-Hermitian structure on $P$ via the Ehresmann connection.
	\begin{definition}
		A linear map $ l_g \colon \frakm \to \frakh $ is given by 
		\begin{equation*}
			\lambda_*(l_g(X))Y = X\times Y, \quad \forall X,Y\in \frakm.
		\end{equation*}
	\end{definition}
	\begin{lemma}\label{lm:lg-Ad-isom}
		The linear map $l_g$ is an $H$-equivariant linear isomorphism, where $H$ acts on $\frakm$ through $\lambda$ and on $\frakh$ through the adjoint representation. More precisely,
		\begin{equation*}
			l_g(\lambda(h)X) = \Ad_h(l_g(X)), \quad \forall h\in H, \forall X\in \frakm.
		\end{equation*}
		Moreover, $l_g$ intertwines the cross product on $\frakm$ with the Lie bracket on $\frakh$:
		\begin{equation}\label{eq:lg-cross}
			l_g(X\times Y ) = [l_gX,l_gY], \quad \forall X,Y\in \frakm.
		\end{equation}
		Thus $l_g$ identifies the $H$-module $\frakm$, equipped with the cross product, with $\frakh$, equipped with its Lie bracket.
	\end{lemma}
	\begin{proof}
		Since $\lambda_* $ is a Lie algebra isomorphism, $l_g\circ \lambda(h) = \Ad_h\circ l_g $ and 
		\begin{equation*}
			\lambda_*\circ l_g \colon\frakm
			\ni X\mapsto * X\in \lambda_*(\frakh)=\mathfrak{so}(\frakm)\simeq \Lambda^2\frakm
		\end{equation*}
		is a linear isomorphism.
		For every $h\in H, $ and $X,Y\in \frakm, $
		\begin{align*}
			\lambda_*(\Ad_hl_g(X))Y & = \lambda(h)\lambda_*(l_g(X))\lambda(h)^{-1}Y\\
			& = \lambda(h)(X\times\lambda(h)^{-1}Y)\\
			& = \lambda(h)X\times Y\\
			& = \lambda_*(l_g\lambda(h)(X))Y.
		\end{align*}
		
		Setting $L_X\coloneqq \lambda_*l_g(X) \in \mathfrak{so}(\frakm), $ by Jacobi identity, for every $X,Y,Z\in \frakm,$
		\begin{align*}
			L_X(Y\times Z) & = L_X Y \times Z+Y\times L_XZ,\\
			[L_X,L_Y](Z) & = X\times(Y\times Z) - Y\times(X\times Z)\\
			& = L_{X\times Y} (Z).
		\end{align*}
		Therefore,
		\begin{align*}
			\lambda_*([l_gX,l_gY])(Z) & = [\lambda_*l_g(X),\lambda_*l_g(Y)](Z)\\
			& = \lambda_*l_g(X\times Y)(Z).
		\end{align*}
		Since $\lambda_*\colon \frakh\to \mathfrak{so}(\frakm)$ is a linear isomorphism, 
		the assertion follows.
	\end{proof}
	Let $\ad P = P\times_{\Ad} \frakh $ be an associated bundle. From \pref{lm:lg-Ad-isom}, there is a bundle isomorphism 
	\begin{equation*}
		l_g\colon TM\ni [u,X] \mapsto [u, l_g(X)]\in \ad P
	\end{equation*}
	
	A direct computation yields the following proposition.
	\begin{proposition}\label{pr:almost-pseudo-Hermitian}
		For every $c>0$, we define a bundle endomorphism
		$\scJ_c\in\Gamma(\End(TP))$ by
		\begin{equation*}
			\scJ_c(X_u^H+U_u^\#) = -c^{-1}(l_g^{-1}U)^H_u+c(l_g X)^\#_u,
		\end{equation*}
		and a pseudo-Riemannian metric $\scG_c$ on $P$ by
		\begin{align*}
			\scG_c(X_u^H+U_u^\#,Y_u^H+V_u^\#)
			&=
			c\,\la X,Y\ra
			+c^{-1}\la l_g^{-1}U,l_g^{-1}V\ra.
		\end{align*}
		Then $(\scG_c,\scJ_c)$ is an almost
		pseudo-Hermitian structure on $P$.
	\end{proposition}
	By \pref{pr:almost-pseudo-Hermitian}, $ (P,\scG_c,\scJ_c) $ is an almost pseudo-Hermitian manifold. Its fundamental two-form $\Omega, $ defined by
	\begin{equation*}
		\Omega(\scrX,\scrY) = \scG_c(\scJ_c\scrX,\scrY), \quad \forall \scrX,\scrY\in \Gamma(TP)
	\end{equation*}
	is independent of $c>0 $ and is given by
	\[ 	\Omega_u(X_u^H+U_u^\#,Y_u^H+V_u^\#) \\
	 = \la X,l_g^{-1}V \ra -\la l_g^{-1}U, Y\ra. \]
	
	\begin{remark}\label{rm:non-symplectic_Lagrangian}
		By \pref{lm:Hor and Ver Lie-bracket}, the fundamental two-form $\Omega$ is non-degenerate but never closed, therefore the corresponding section cannot be symplectic Lagrangian but almost symplectic Lagrangian.
		Indeed, for $\Xhat,\Yhat,\Zhat \in C^\infty(P,\frakm)^H $ and $ \Uhat,\Vhat,\What \in C^\infty(P,\frakh)^H, $
		\begin{align*}
			d\Omega (\Xhat^H,\Yhat^H,\Zhat^H) & = -\sum_{\text{cyc}}{\Omega([\Xhat^H,\Yhat^H],\Zhat^H)}\\
			& = \sum_{\text{cyc}}{\Omega(\widetilde\Omega(\Xhat^H,\Yhat^H)^\#,\Zhat^H)}\\
			& = -\sum_{\text{cyc}}{\la l_g^{-1}\widetilde\Omega(\Xhat^H,\Yhat^H), \Zhat \ra},\\
			d\Omega (\Xhat^H,\Yhat^H,\Uhat^\#) & = \Xhat^H\Omega(\Yhat^H,\Uhat^\#) - \Yhat^H\Omega(\Xhat^H,\Uhat^\#)\\
			& \quad  - \Omega([\Xhat^H,\Yhat^H],\Uhat^\#) - \Omega([\Yhat^H,\Uhat^\#], \Xhat^H) - \Omega([\Uhat^\#,\Xhat^H],\Yhat^H)\\
			& = \Xhat^H\la l_g^{-1}\Uhat, \Yhat \ra - \Yhat^H\la l_g^{-1}\Uhat , \Xhat\ra \\
			& \quad -\la l_g^{-1}\Uhat,\Xhat^H(\Yhat) - \Yhat^H(\Xhat) \ra + \la l_g^{-1}\Yhat^H(\Uhat),\Xhat \ra -\la l_g^{-1}\Xhat^H(\Uhat),\Yhat \ra\\
			& =0,\\
			d\Omega (\Xhat^H,\Uhat^\#,\Vhat^\#) & = - \Uhat^\#\Omega(\Xhat^H,\Vhat^\#) + \Vhat^\#\Omega(\Xhat^H,\Uhat^\#) - \Omega([\Uhat^\#,\Vhat^\#], \Xhat^H)\\
			& = -\la l_g^{-1}\Uhat\times l_g^{-1}\Vhat,\Xhat \ra ,\\
			d\Omega (\Uhat^\#,\Vhat^\#,\What^\#) & = 0.
		\end{align*}
		Thus, the term of $d\Omega (\Xhat^H,\Uhat^\#,\Vhat^\#)$ does not vanish for some $\Xhat,\Uhat,\Vhat.$
	\end{remark}
	
	\begin{lemma}\label{lm:Hor and Ver Lie-bracket}
		Let $ \Xhat,\Yhat \in C^\infty(P,\frakm)^H $ and $ \Uhat, \Vhat \in C^\infty(P,\frakh)^H $ be $H$-equivariant function.
		The following formulae hold.
		\begin{align*}
			\ld \Xhat^H,\Yhat^H\rd & = \lb \Xhat^H(\Yhat) - \Yhat^H(\Xhat) \rb^H - \widetilde\Omega(\Xhat^H,\Yhat^H)^\#,\\
			\ld \Uhat^\#,\Yhat^H \rd & = -\lb \Yhat^H(\Uhat)\rb^\#,\\
			\ld \Uhat^\#,\Vhat^\# \rd & = -[\Uhat,\Vhat]^\#.
		\end{align*}
		\begin{proof}
			These formulae
			\[ 0 = d\theta+\lambda_*\omegatilde\wedge\theta, \quad \widetilde\Omega = d\omegatilde + \omegatilde\wedge\omegatilde \]
			lead our assertions by the following direct calculations
			\begin{align*}
				0 =d\theta(\Xhat^H,\Yhat^H) & = \Xhat^H(\theta(\Yhat^H)) - \Yhat^H(\theta(\Xhat^H)) -\theta([\Xhat^H,\Yhat^H]),\\
				\theta([\Xhat^H,\Yhat^H]) & = \Xhat^H(\Yhat) - \Yhat^H(\Xhat),\\
				\widetilde\Omega(\Xhat^H,\Yhat^H) & = d\omegatilde(\Xhat^H,\Yhat^H) = -\omegatilde([\Xhat^H,\Yhat^H]). 
			\end{align*}
			Therefore, 
			\begin{equation}\label{eq:hori-bracket}
				\ld \Xhat^H,\Yhat^H\rd = \lb \Xhat^H(\Yhat) - \Yhat^H(\Xhat) \rb^H - \widetilde\Omega(\Xhat^H,\Yhat^H)^\#.
			\end{equation}
			The other formulae are also obtained by the similar calculations.
		\end{proof}
	\end{lemma}
	The Nijenhuis tensor $N_{\scJ_c}$ of the almost complex structure $\scJ_c$ is defined as follows
	\[ N_{\scJ_c}(\scrX,\scrY) = [\scJ_c(\scrX), \scJ_c(\scrY)] -\scJ_c[\scJ_c\scrX,\scrY] - \scJ_c[\scrX,\scJ_c\scrY] -[\scrX,\scrY],\quad  \forall \scrX,\scrY\in \Gamma(TP). \]
	\begin{lemma}\label{lm:cal-Nijenhuis}
		Let $\scrX, \scrY\in \Gamma(TP)^H$ be vector fields 
		\begin{align}
			\scrX & = \theta(\scrX)^H + \omegatilde(\scrX)^\#, \quad  \Xhat = \theta(\scrX) \in C^\infty(P,\frakm)^H, \quad \omegatilde(\scrX)=\Uhat \in C^\infty(P,\frakh)^H, \notag\\
			\scrY & = \theta(\scrY)^H + \omegatilde(\scrY)^\#,\quad  \Yhat = \theta(\scrY) \in C^\infty(P,\frakm)^H, \quad \omegatilde(\scrY)=\Vhat \in C^\infty(P,\frakh)^H.\notag
		\end{align}
		Then we have 
		\begin{align}
			N_{\scJ_c}(\Xhat^H,\Yhat^H) & = -c^2[l_g\Xhat,l_g\Yhat]^\# + \widetilde\Omega(\Xhat^H,\Yhat^H)^\#,\notag\\
			N_{\scJ_c}(\Xhat^H,\Vhat^\#) & = -l_g^{-1
			}\lc [l_g\Xhat, \Vhat] - c^{-2}\widetilde\Omega\lb \Xhat^H,(l_g^{-1}\Vhat)^H\rb \rc^H ,\label{eq:N_J-cal}\\
			N_{\scJ_c}(\Uhat^\#,\Vhat^\#) & = [\Uhat,\Vhat]^\# - c^{-2}\widetilde\Omega\lb (l_g^{-1}\Uhat)^H,(l_g^{-1}\Vhat)^H\rb^\#.\notag
		\end{align}
	\end{lemma}
	\begin{proof}
		By \pref{eq:hori-bracket},
		\begin{align*}
			N_{\scJ_c}\lb \Xhat^H,\Yhat^H \rb & = c^2[(l_g\Xhat)^\#, (l_g\Yhat)^\#] - \scJ_c[c(l_g\Xhat)^\#,\Yhat^H] - \scJ_c[\Xhat^H,c(l_g\Yhat)^\#] -[\Xhat^H,\Yhat^H],\\
			& =  c^2[(l_g\Xhat)^\#, (l_g\Yhat)^\#] - \lb l_g^{-1}\Yhat^H(l_g\Xhat) \rb^H + \lb l_g^{-1}\Xhat^H(l_g\Yhat) \rb^H -[\Xhat^H,\Yhat^H],\\
			& = -c^2[l_g\Xhat,l_g\Yhat]^\# +\widetilde\Omega\lb \Xhat^H,\Yhat^H \rb^\#.
		\end{align*}
		The other components are obtained by similar calculations.
	\end{proof}
	Newlander and Nirenberg \cite{newlander1957complex} showed that an arbitrary almost complex structure is integrable if and only if its Nijenhuis tensor vanishes.
	Since \pref{eq:N_J-cal} are exactly the same condition under $l_g$, we have the following theorem.
	\begin{theorem}\label{th:integrability-complex}
		The almost complex structure $\scJ_c$ on $P$ is a complex structure if and only if $ (M,g) $ has constant sectional curvature $(-1)^{r-1}c^2$.
	\end{theorem}
	\begin{proof}
		By \pref{lm:cal-Nijenhuis} and Newlander--Nirenberg theorem, the almost complex structure $\scJ_c$ is integrable, if and only if 
		\[ \widetilde\Omega(\Xhat^H,\Yhat^H) = c^2[l_g\Xhat, l_g\Yhat], \quad \forall \Xhat,\Yhat\in C^\infty(P,\frakm)^H. \]
		By \pref{eq:lg-cross} in \pref{lm:lg-Ad-isom}, for every section $s \in \Gamma(\scU,P), $ and put $ \scX = [s, X], X,Y,Z\in \frakm$
		\begin{align*}
			R(\scX,\scY)\scZ & = \ld s,\lambda_*\lb s^*\widetilde\Omega\lb \scX,\scY \rb \rb Z \rd,\\
			& = \ld s,c^2[\lambda_*l_gX,\lambda_*l_gY]Z \rd,\\
			& = \ld s,c^2(X\times Y)\times Z \rd,\\
			& = (-1)^rc^2 \lb g(\scX,\scZ)\scY - g(\scY,\scZ)\scX \rb.
		\end{align*}
		Therefore, the sectional curvature $\kappa$ of base manifold $(M,g)$ spanned by $\{\scX,\scY\}$ is
		\[ \kappa = \frac{g(R(\scX,\scY)\scY,\scX)}{g(\scX,\scX)g(\scY,\scY)-g(\scX,\scY)^2} = (-1)^{r-1}c^2. \]
	\end{proof}
	
	An almost pseudo-Hermitian manifold $(M,g,\scJ)$ is called nearly pseudo-K\"ahler if
	\[
	(\nabla_\scX \scJ)\scX=0
	\]
	for all vector fields $\scX\in\Gamma(TM)$, where $\nabla$ is the Levi-Civita
	connection of $(M,g)$. Equivalently,
	\[
	(\nabla_\scX \scJ)\scY+(\nabla_\scY \scJ)\scX=0
	\]
	for all vector fields $\scX,\scY$.
	The following Levi-Civita connection formulae follow from the Koszul formula.
	\begin{lemma}
		Let $\nablabar$ be the Levi-Civita connection of $(P,\mathcal G_c)$.
		For $\widehat X,\widehat Y\in C^\infty(P,\frakm)^H$ and
		$\widehat U,\widehat V\in C^\infty(P,\mathfrak h)^H$, the horizontal and
		vertical components of $\nablabar$ are given by
		\begin{align*}
			\nablabar_{\Xhat^H}\Yhat^H & = \lb \Xhat^H(\Yhat) \rb^H -\frac{1}{2}\widetilde\Omega\lb \Xhat^H,\Yhat^H \rb^\#,\\
			\nablabar_{\Xhat^H}\Vhat^\# & = \scT_c(\Xhat,\Vhat)^H + \lb \Xhat^H(\Vhat) \rb^\#,\\
			\nablabar_{\Uhat^\#}\Yhat^H & = \scT_c(\Yhat,\Uhat)^H,\\
			\nablabar_{\Uhat^\#}\Vhat^\# & = -\frac{1}{2}[\Uhat,\Vhat]^\#,
		\end{align*}
		where $\scT_c(\Xhat,\Vhat) \in C^\infty(P,\frakm)^H $ is given by
		\[ \la \scT_c(\Xhat, \Vhat),\Zhat\ra = \frac{1}{2c^2}\la l_g^{-1}\widetilde\Omega\lb \Xhat^H,\Zhat^H\rb, l_g^{-1}\Vhat\ra. \]
		
	\end{lemma}

	\begin{theorem}\label{th:NK-str}
		The almost pseudo-Hermitian structure $ (\scG_c, \scJ_c) $ on $P$ is nearly pseudo-K\"ahler if and only if $(M,g)$ has constant sectional curvature $\frac{(-1)^{r}}{3}c^2.$
	\end{theorem}
	
	\begin{proof}
		Assume that $(P,\scG_c,\scJ_c)$ is nearly pseudo-K\"ahler. Then covariant derivatives are given by
		\begin{align*}
			(\nablabar_{\Xhat^H}\scJ_c)\Xhat^H & = c\scT_c(\Xhat,l_g\Xhat)^H,\\
			(\nablabar_{\Uhat^\#}\scJ_c)\Uhat^\# & = -c^{-1}\scT_c(l_g^{-1}\Uhat, \Uhat)^H,\\
			(\nablabar_{\Uhat^\#}\scJ_c)\Xhat^H + (\nablabar_{\Xhat^H}\scJ_c)\Uhat^\# & = -2cl_g\scT_c(\Xhat,\Uhat)^\# + \frac{1}{2c}\widetilde\Omega(\Xhat^H,(l_g^{-1}\Uhat)^H)^\# - \frac{1}{2}c[\Uhat,l_g\Xhat]^\#
		\end{align*}
		all vanishes.
		For every $\frakm$-valued two form $\alpha \in \Lambda^2\frakm^*\otimes \frakm $, there exists a unique $S\in \End(\frakm) $ via the isomorphism $\Lambda^2\frakm \ni X\wedge Y \mapsto X\times Y\in \frakm$
		\begin{equation*}
			\alpha(X,Y) = S(X\times Y).
		\end{equation*}
		Hence, there exists a unique $\scS \in C^\infty(P,\End(\frakm))^H, $
		\begin{equation*}
			l_g^{-1}\widetilde\Omega(\Xhat^H,\Yhat^H) = \scS(\Xhat\times\Yhat).
		\end{equation*}
		If $(P,\scG_c,\scJ_c)$ is nearly pseudo-K\"ahler, then $ \scT_c(\Xhat,l_g\Xhat) = 0 $ that implies
		\[ \la \scT_c(\Xhat, l_g\Xhat),\Zhat\ra = \frac{1}{2c^2}\la \scS(\Xhat\times \Zhat), \Xhat\ra = 0, \quad \forall \Xhat, \Zhat \in C^\infty(P,\frakm)^H. \]
		Then there exists a unique function $\rho\in C^\infty(P,\bR), $
		\[ \scS = \rho\id_\frakm. \]
		From $ \la X\times Y,Z\ra = -\la X\times Z,Y \ra, $
		\begin{align*}
			\la \scT_c(\Xhat,\Uhat),\Zhat\ra & = \frac{1}{2c^2}\la l_g^{-1}\widetilde\Omega\lb \Xhat^H,\Zhat^H \rb,l_g^{-1}\Uhat \ra\\
			& = -\frac{\rho}{2c^2}\la \Xhat\times l_g^{-1}\Uhat, \Zhat \ra.
		\end{align*}
		then
		\[ \scT_c(\Xhat,\Uhat) = -\frac{\rho}{2c^2}\lb \Xhat\times l_g^{-1}\Uhat\rb. \]
		By \pref{eq:lg-cross} in \pref{lm:lg-Ad-isom}, we have
		\begin{align*}
			(\nablabar_{\Uhat^\#}\scJ_c)\Xhat^H + (\nablabar_{\Xhat^H}\scJ_c)\Uhat^\# & = -2cl_g\scT_c(\Xhat,\Uhat)^\# + \frac{1}{2c}\widetilde\Omega(\Xhat^H,(l_g^{-1}\Uhat)^H)^\# - \frac{1}{2}c[\Uhat,l_g\Xhat]^\#\\
			& = \lb \frac{3\rho}{2c}l_g\lb \Xhat\times l_g^{-1}\Uhat\rb -\frac{c}{2}[\Uhat,l_g\Xhat] \rb^\#\\
			& = \frac{3\rho+c^2}{2c}l_g \lb \Xhat\times l_g^{-1}\Uhat \rb^\#=0.
		\end{align*}
		Then $\rho = -\frac{c^2}{3}.$
		It is also followed that if $\rho = -\frac{c^2}{3}, $ for every section $s \in \Gamma(\scU,P), $ and put $\scX = [s, X], X,Y,Z\in \frakm,$ then
		\begin{align*}
			R(\scX,\scY)\scZ & = \ld s,\lambda_*\lb s^*\widetilde\Omega\lb \scX,\scY \rb \rb Z \rd,\\
			& = \ld s,\rho(X\times Y)\times Z \rd,\\
			& = \frac{(-1)^rc^2}{3} \lb g(\scY,\scZ)\scX - g(\scX,\scZ)\scY \rb.
		\end{align*}
		Thus, the sectional curvature of the base manifold is
		\[ \kappa = \frac{(-1)^r}{3}c^2. \]
		Conversely, if $(M,g)$ admits constant sectional curvature $\kappa = \frac{(-1)^rc^2}{3}, $ 
		\[ 	\widetilde\Omega\lb \Xhat^H,\Yhat^H \rb = -\frac{c^2}{3}\ld l_g\Xhat,l_g\Yhat\rd. \]
		It implies $(\nablabar_\scrX \scJ_c) \scrX = 0 $ for all $\scrX \in \Gamma(TP).$
	\end{proof}
	
	
	\subsection{Unit-length generalized Killing spinors and Lagrangian sections}\label{sc:unit-GKS-Lagrangian}
	The spinor bundle $\Sigma M$ is realized as the associated bundle $ P\times_\Delta \bC^2, $ where $\Delta\colon H\hookrightarrow\GL(2,\bC)$ is the defining
	two-dimensional complex representation of $H$, identified with the
	spin representation in dimension three.
	In this notation, a spinor field $\psi \in \Gamma(\Sigma M) $ is represented by an $H$-equivariant map $\varphi\in C^\infty(P,\bC^2)^H $ as follows:
	\begin{equation*}
		\psi(x) = [u,\varphi(u)],\quad \forall u\in \pi^{-1}(x). 
	\end{equation*}
	Clifford multiplication is given by the linear map
	\begin{equation*}
		\gamma\colon \frakm\ni X_i\mapsto \gamma_i\in (\sqrt{-1})^r\frakh\subset\mathfrak{gl}(2,\bC)\simeq \End(\bC^2),
	\end{equation*}
	where $\{X_i\}_{i=0}^2$ and $ \{\gamma_i\}_{i=0}^2 $ are orthonormal bases of $\frakm$ and $(\sqrt{-1})^r\frakh.$
	Then
	\begin{equation*}
		X\cdot \varphi(u) = \gamma(X)\varphi(u), \quad \scX\cdot\psi(x) = [u, \iota_u^{-1}\scX_x\cdot\varphi(u)], \quad \forall u\in \pi^{-1}(x).
	\end{equation*}
	With respect to these bases, we use the following gamma
	matrices.
	\begin{enumerate}
		\item In the Riemannian case,
		\begin{equation*}
			\gamma_0 = \begin{pmatrix}
				\sqrt{-1} & 0\\0 &-\sqrt{-1}
			\end{pmatrix}, \quad
			\gamma_1 = \begin{pmatrix}
				0 & 1\\-1 & 0 
			\end{pmatrix}, \quad
			\gamma_2 = \begin{pmatrix}
				0 & \sqrt{-1}\\\sqrt{-1} & 0
			\end{pmatrix},
		\end{equation*}
		\item In the Lorentzian case,
		\begin{equation*}
			\gamma_0 = \begin{pmatrix}
				1 & 0\\0 &-1
			\end{pmatrix}, \quad
			\gamma_1 = \begin{pmatrix}
				0 & 1\\-1 & 0 
			\end{pmatrix}, \quad
			\gamma_2 = \begin{pmatrix}
				0 & \sqrt{-1}\\\sqrt{-1} & 0
			\end{pmatrix}.		\end{equation*}
	\end{enumerate}
	
	\begin{lemma}\label{lm:lg-gamma}
		For every $X\in \frakm,$
		\begin{equation}\label{eq:lg-gamma}
			l_g(X)=\frac{(\sqrt{-1})^r}{2}\gamma(X).
		\end{equation}
	\end{lemma}
	\begin{proof}
		Let $\{X_i\}_{i=0}^2 $ be an orthonormal basis of $\frakm,$ and put $\gamma_i=\gamma(X_i).$
		By the definition of $l_g, $
		\begin{equation*}
			\lambda_* \lb l_g(X_i)\rb(X_j) = X_i \times X_j = \epsilon_{ij}{}^k X_k.
		\end{equation*}
		Derivative of $\gamma(\lambda(h)X) = h\gamma(X)h^{-1} $ yields
		\begin{equation*}
			\gamma\lb\lambda_*(U)X\rb
			=
			[U,\gamma(X)],
			\quad U\in\frakh,\quad X\in\frakm.
		\end{equation*}
		Moreover, since $ \frac{(\sqrt{-1})^r}{2}\gamma_i\in \frakh, $
		\begin{align*}
			\ld \frac{(\sqrt{-1})^r}{2}\gamma_i, \gamma_j \rd 
			& = \epsilon_{ij}{}^k\gamma_k,\\
			\gamma \lb \lambda_*\lb \frac{(\sqrt{-1})^r}{2}\gamma_i\rb X_j \rb  & = \gamma\lb X_i\times X_j \rb.
		\end{align*}
		It implies that 
		\[ l_g(X_i) = \frac{(\sqrt{-1})^r}{2}\gamma_i. \]
		The assertion now follows by linearity.
	\end{proof}
	
	The spinor bundle admits a nondegenerate Hermitian form:
	\begin{equation*}
		\la \psi,\phi \ra =  \psi^\dagger(\gamma_0)^r \phi,
	\end{equation*}
	which is positive definite for $r=0$ and indefinite for $r=1$.
	Then, for each $k=0,1,2,$
	\begin{align*}
		\la (\sqrt{-1})^r\gamma_k\psi,\psi\ra & = (-\sqrt{-1})^r\psi^\dagger\gamma_k^\dagger(\gamma_0)^r\psi\\
		& = -(\sqrt{-1})^r\psi^\dagger(\gamma_0)^r\gamma_k\psi\\
		& = -\la \psi,(\sqrt{-1})^r\gamma_k\psi\ra.
	\end{align*}
	
	Set $\scE_0=\{+1\}$ and
	$\scE_1=\{+1,-1\}$. For $\epsilon\in\scE_r$, a
	spinor field satisfying $\la\psi,\psi\ra=\epsilon$ is called an
	$\epsilon$-unit spinor, or simply a spinor field of unit length.
	In the Lorentzian case, it is spacelike for $\epsilon=+1$ and
	timelike for $\epsilon=-1$.
	
	For an $\epsilon$-unit spinor $\psi$, the matrices
	above show that the three spinors
	$(\sqrt{-1})^r\gamma_k\psi$ form a real basis of
	$\{\phi\in\bC^2:\operatorname{Re}\la\phi,\psi\ra=0\}$.
	Since $0=\nabla_X\la\psi,\psi\ra, $ every spinor field $\psi$ of unit length has a unique real $(1,1)$-type tensor field $A$ such that
	\begin{equation}\label{eq:GKS}
		\nabla_\scX \psi =(\sqrt{-1})^rA(\scX)\cdot\psi,\quad \forall \scX\in \Gamma(TM).
	\end{equation}
	
	\begin{definition}
		For every $\epsilon\in\scE_r$, an
		$\epsilon$-unit spinor field $\psi$ is called a
		generalized Killing spinor if the unique tensor field $A$ in
		\pref{eq:GKS} is symmetric with respect to $g$.
	\end{definition}
	
	\begin{remark}\label{rm:null-spinor}
		The non-null assumption is essential in the
		Lorentzian case. For a nowhere-vanishing null spinor field $\psi$,
		its annihilator
		$\operatorname{Ann}(\psi)=\{V\in TM:V\cdot\psi=0\}$ is a null line
		bundle. Hence a tensor field $A$ satisfying \eqref{eq:GKS}, even if it
		exists, is not unique. Indeed, if $K$ is a local nonzero section of
		$\operatorname{Ann}(\psi)$ and $A$ is symmetric, then
		$A+fK^\flat\otimes K$ is another symmetric solution for every smooth
		function $f$. We therefore exclude null spinors from the
		correspondence below.
	\end{remark}
	
	Hereafter, we assume that every $\pi\colon P\to M$ is trivial.
	For $\epsilon\in\scE_r$, every
	$\epsilon$-unit spinor field $\psi \in \Gamma(\Sigma M) $ is described by an $H$-equivariant
	function $F\in C^\infty(P,H)^H$ as follows:
	\begin{equation}\label{eq:unit-length-spinor}
		\psi_\epsilon(x)
		=[u,F(u)\mathbb{1}_\epsilon],\qquad
		\mathbb{1}_{+1}=\begin{pmatrix}1\\0\end{pmatrix},\quad
		\mathbb{1}_{-1}=\begin{pmatrix}0\\1\end{pmatrix}.
	\end{equation}
	For $r=0$, only $\epsilon=+1$ occurs. For $r=1$,
	$\psi_{+1}$ is spacelike and $\psi_{-1}$ is timelike; the same
	function $F$ determines one spinor field of each causal type.
	\begin{remark}
		For each $\epsilon\in\scE_r$, a spin manifold $(M^{r,3-r},\eta)$ admits an $\epsilon$-unit spinor field if and only if there is an $H$-equivariant function $ F\in C^\infty(P,H)^H $ that requires $\pi\colon P\to M $ is trivial. Although the orthonormal frame bundle of every oriented three-dimensional Riemannian manifold is trivial, as a consequence of parallelizability, the corresponding statement does not hold in general for Lorentzian three-manifolds.
	\end{remark}
	Every spinor of unit length defined by \pref{eq:unit-length-spinor} determines an oriented orthonormal frame
	$\{\xi_k=[\bullet, \lambda(F)X_k]\}_{k=0}^2$ characterized by
	\begin{equation*}
		\xi_k\cdot \psi_\epsilon = [\bullet, F\gamma_k\mathbb{1}_\epsilon], \quad k=0,1,2.
	\end{equation*}
	We obtain the pseudo-Riemannian generalization of \cite[Lemma~2.1]{moroianu2014generalized} with the same argument.
	\begin{lemma}
		A spinor field $\psi \in \Gamma(\Sigma M) $ of unit length given by
		\[
		\psi(x) =[u, F(u)\mathbb{1}_\epsilon],
		\quad \forall u\in \pi^{-1}(x).
		\]
		is a generalized Killing spinor if and only if the characteristic frame $\{\xi_k\}_{k=0}^2$ of $\psi$
		given by
		\begin{equation}\label{eq:characteristic-frame-GKS}
			\xi_k(x) = [u, \lambda(F(u))X_k], \quad \forall u\in \pi^{-1}(x)  
		\end{equation}
		is divergence-free.
	\end{lemma}
	
	\begin{lemma}\label{lm:scM-def}
		For each $u\in P, $ define a linear map
		\begin{equation*}
			\scM_u \colon \frakm \ni X\mapsto dF_u(X^H_u)F(u)^{-1}\in \frakh. 
		\end{equation*}
		Then, we have
		\begin{equation}\label{eq:scM-gauge}
			\scM_{uh}(X) = \Ad_{h^{-1}}(\scM_u(\lambda(h)X)),
		\end{equation}
		for every $u\in P, h\in H, $ and $ X\in \frakm. $
	\end{lemma}
	\begin{proof}
		From the uniqueness of horizontal lifting,
		\begin{align*}
			\pi_*{R_h}_*\lb \lambda(h)X \rb^H_u & =\pi_*\lb \lambda(h)X\rb^H_u\\
			& = [u,\lambda(h)X]\\
			& = [uh,X],
		\end{align*}
		we have
		\begin{equation*}
			{R_h}_*\lb \lambda(h)X \rb^H_u = X^H_{uh}.
		\end{equation*}
		Therefore, 
		\begin{align*}
			\scM_{uh}(X) & = dF_{uh}(X^H_{uh})F(uh)^{-1}\\
			& = dF_{uh}({R_h}_*\lb \lambda(h)X \rb^H_u)F(u)^{-1}h\\
			& = h^{-1}dF_u(\lb \lambda(h)X\rb^H_u)F(u)^{-1}h\\
			& = \Ad_{h^{-1}}\scM_u(\lambda(h)X).
		\end{align*}
	\end{proof}

	Since $\Gamma(P)$ is canonically identified with $ C^\infty(P,H)^H $ as a $C^\infty(M,H)$-torsor, 
	the both right actions of $C^\infty(M,H) $ are given by
	\begin{equation}\label{eq:section-torsor}
		s\cdot f(x) = s(x)f(x), \quad F\cdot f(u)= F(u)f(x), \quad  f\in C^\infty(M,H).
	\end{equation}
	Fix a global reference section $s_0\in\Gamma(P)$ and use the
	induced trivialization $P\simeq M\times H$, under which
	$s_0(x)=(x,\mathbf1)$. Every
	$F\in C^\infty(P,H)^H$ and $\epsilon$-unit spinor are uniquely determined by
	$f\coloneqq F\circ s_0\colon M\to H$ through
	\begin{equation*}
		F(s_0(x)h)=h^{-1}f(x), \quad \psi_{\epsilon}(x) = [s_0(x),f(x)\mathbb{1}_\epsilon].
	\end{equation*}
	The corresponding section is $s_F=s_0\cdot f$.
	
	\begin{lemma}\label{lm:F-f-scM}
		Let $F \in C^\infty(P,H)^H $ be fixed, and let $\scM $ be defined by $F$ as in \pref{lm:scM-def}.
		For every $\sigma\colon M\to P, $ and $\scX \in \Gamma(TM), $ we define
		\[ f_\sigma = \sigma^*F\in C^\infty(M,H), \quad X_\sigma\coloneqq (\sigma^*\theta)(\mathcal X) \in C^\infty(M,\frakm).  \]
		Then, for an arbitrary $ x\in M, $ we have
		\[ \scM_{\sigma(x)}(X_\sigma) = df_\sigma|_x(\scX)f_\sigma(x)^{-1} + (\sigma^*\omegatilde)_x(\scX). \]
		In particular, if we choose $ s_0 $ as a global section $ f\coloneqq s_0^*F, $ and $ s_F $ as the section determined by $F ,$ i.e.
		\[ s_F(x) = s_0(x)F(s_0(x)), \]
		we have
		\begin{align}
			\scM_{s_0(x)}(X_0) & = df_x(\scX)f(x)^{-1}+(s_0^*\omegatilde)_{x}(\scX),\label{eq:scM1}\\
			\scM_{s_F(x)}(X_F) & = (s_F^*\omegatilde)_x(\scX), \label{eq:scM2}
		\end{align}
		where $X_0\coloneqq X_{s_0}, $ and $ X_F \coloneqq X_{s_F}.$
	\end{lemma}
	\begin{proof}
		We put
		\begin{align*}
			\sigma_*\scX = X_\sigma^H + U_\sigma^\#\in \Gamma(\sigma^*TP), \quad U_\sigma = \sigma^*\omegatilde(\scX) \in C^\infty(M,\frakh).
		\end{align*}
		Since by $H$-equivariance
		\[ dF_{\sigma(x)} (U_\sigma^\#(x)|_{\sigma(x)})F(\sigma(x))^{-1} = -U_\sigma(x)= -(\sigma^*\omegatilde)_x(\scX), \]
		then
		\begin{align*}
			\scM_{\sigma(x)}(X_\sigma) & = dF_{\sigma(x)}(X_\sigma^H)F(\sigma(x))^{-1}\\
			& = dF_{\sigma(x)}(\sigma_*\scX)F(\sigma(x))^{-1} + (\sigma^*\omegatilde)_x(\scX)\\
			& = df_\sigma|_x(\scX)f_\sigma(x)^{-1} + (\sigma^*\omegatilde)_x(\scX).
		\end{align*}
	\end{proof}
	The equivariance relation in \pref{lm:scM-def} defines a bundle homomorphism
	$\bfM \in \Gamma(T^*M\otimes \ad P)$ given by
		\begin{equation}\label{eq:scM-section}
		\bfM\colon TM\ni \scX|_x \mapsto [u,\scM_u(X)] \in \ad P,\quad \forall u\in \pi^{-1}(x), \ \scX|_x =[u,X]\in T_xM.
	\end{equation} 
	
	\begin{proposition}\label{pr:GKS-condition}
		A unit-length spinor $\psi_\epsilon = [\bullet, F\mathbb{1}_\epsilon]$ is a generalized Killing spinor if and only if $l_g^{-1}\circ \bfM$ is symmetric, i.e.
		\begin{equation}\label{eq:symmetric-scM}
			\la l_g^{-1} \scM_u(X), Y\ra = \la X,l_g^{-1}\scM_u(Y)\ra, \quad \forall u\in P, \forall X,Y\in \frakm.
		\end{equation} 
		In particular, the symmetric endomorphism $A\in \Gamma(\End(TM))$ in \pref{eq:GKS}
		is given by
		\begin{equation*}
			A = \frac{1}{2}l_g^{-1}\bfM.
		\end{equation*}
	\end{proposition}
	\begin{proof}
		For any global section $s\in\Gamma(P)$, the $\epsilon$-unit spinor
		$\psi_\epsilon$ defined by $F\in C^\infty(P,H)^H$ is represented by
		\[
		\psi_\epsilon(x)=[s(x),f(x)\mathbb{1}_\epsilon],
		\qquad f\coloneqq s^*F\in C^\infty(M,H),
		\qquad \forall x\in M.
		\]
		By \pref{lm:F-f-scM}, and \eqref{eq:lg-gamma} for every $x\in M$ we have
		\begin{align*}
			\nabla_\scX\psi_\epsilon|_x & = \ld s(x), \lc df_{x}(\scX)f(x)^{-1} + (s^*\omegatilde)_x(\scX) \rc f(x)\mathbb{1}_\epsilon \rd\notag\\
			& = \ld s(x), \scM_{s(x)}(X) f(x)\mathbb{1}_\epsilon \rd \qquad \lb X=(s^*\theta)_x(\scX) \rb  \\
			& = \frac{(\sqrt{-1})^r}{2}\lb l_g^{-1}\bfM(\scX) \rb \cdot \psi_\epsilon(x)\\
			& = (\sqrt{-1})^rA(\scX)\cdot \psi_\epsilon(x).
		\end{align*}
		Therefore, by \pref{lm:scM-def} and \eqref{eq:scM-section}, $\psi_\epsilon$ is a generalized Killing spinor if and only if $l_g^{-1}\bfM$ is symmetric, or equivalently, if and only if $ l_g^{-1}\scM_u$ is symmetric for every $u\in P $.
	\end{proof}
	There is a one-to-one correspondence between $\epsilon$-unit generalized Killing spinors on $(M,g) $ and Lagrangian sections on $ (P,\Omega). $ A section $s\in \Gamma(P)$ is called Lagrangian if $s^*\Omega = 0.$
	\begin{proposition}\label{pr:Lag-sym}
		Let $s_F \in \Gamma(P) $ be a section, for an $H$-equivariant function $F\in C^\infty(P,H)^H, $
		defined by
		\begin{equation*}
			s_F\colon M\ni x \mapsto uF(u)\in P, \quad \forall u\in \pi^{-1}(x).
		\end{equation*} 
		The section $ s_F $ is Lagrangian of $(P,\Omega) $ if and only if $l_g^{-1}
		\circ \scM_u $ is symmetric for all $u\in P$.
	\end{proposition}
	
	\begin{proof}
		By \eqref{eq:scM2} in \pref{lm:F-f-scM}, and \eqref{eq:symmetric-scM} in \pref{pr:GKS-condition},
		\begin{align*}
			ds_F(\scX) & = X_F^H+\scM_{s_F}(X_F)^\#,\quad ds_F(\scY) = Y^H_F+\scM_{s_F}(Y_F)^\#,\\
			s_F^*\Omega(\scX,\scY)& = \la X_F,l_g^{-1}\scM_{s_F}(Y_F) \ra - \la Y_F, l_g^{-1}\scM_{s_F}(X_F) \ra.
		\end{align*}
		It implies that $s_F^*\Omega = 0 $ is equivalent to the symmetry of the operator $l_g^{-1}\circ\scM_u $ for each $ u\in P$ by \pref{lm:scM-def}.
	\end{proof}
	Therefore, we obtain the main theorem \ref{th:GKS-LagGraph-main} by \pref{pr:GKS-condition} and \ref{pr:Lag-sym}.
	
	\begin{corollary}\label{cr:constant-map-GKS}
		Fix a section $s\colon M\to P$ and let
		\[
		\scX_k=[s,X_k],\qquad k=0,1,2,
		\]
		be the associated oriented orthonormal frame. For $h_0\in H$, define
		\[
		F(s(x)h)=h^{-1}h_0.
		\]
		Then the following conditions are equivalent:
		\begin{enumerate}
			\item the spinor represented by $F$ is a generalized Killing
			spinor,
			\item the section $s_F=sh_0$ is Lagrangian,
			\item the linear map
			\[
			\frakm\ni X\longmapsto
			l_g^{-1}\bigl((s^*\widetilde\omega)(\scX)\bigr)
			\in\frakm, \quad (\scX = [s,X])
			\]
			is symmetric at every point of $M$,
			\item The orthonormal frame $\{\scX_k\}_{k=0}^2$ is divergence-free.
		\end{enumerate}
		These conditions are independent of $h_0$.
	\end{corollary}
	
	\section{Constant-curvature spaces}
	\label{sc:constant-curvature spaces}
	We now consider the three-dimensional homogeneous space form $M=G/H, $ equipped with the indicated $G$-invariant spin structures,
where $P=G$ is a double cover of $\operatorname{Isom}_0(M,g)$ determined by the spin structure. Under this identification, $g$ is $G$-invariant and
	the projection
	\begin{equation*}
		\pi\colon P=G\to M=G/H 
	\end{equation*}
	is a principal $H$ bundle and defines a $G$-invariant spin structure on $M.$
	We say that a spin structure on $(M,g)$ is $G$-invariant if the isometry group action of $G$ on $M$ lifts to the principal spin bundle $P $ compatibly with the spin structure. The lifted action induces an action of $G$ on a spinor bundle, and a spinor is called $G$-invariant if it is fixed by this induced action. For further details on invariant spin structures and spinors, see \cite{daura2022g, agricola2023invariant}.
	The flat space forms and the two Lie group space forms
	$\bS^3$ and $\AdS_3$ will be treated simultaneously below.
	
	The remaining simply connected non-flat space forms,
	\[
	\mathbb H^3=\SL(2,\bC)/\SU(2),
	\qquad
	\dS_3=\SL(2,\bC)/\SU(1,1),
	\]
	are also treated in this setting. By \pref{th:integrability-complex}, in both cases $\SL(2,\bC)$ have integrable complex structures $\scJ_c $ when their constant sectional curvatures are $ \kappa = (-1)^{r-1}c^2.$  However, there are no spinors of unit length on $\dS_3, $ since the principal spinor bundle $\pi\colon \SL(2,\bC) \to \dS_3 $ is not trivial. Hence, there is no $H=\SU(1,1)$-equivariant function $ F \in C^\infty(\SL(2,\bC),\SU(1,1))^{\SU(1,1)}. $
	
	For $\bH^3$, the symmetry of
	$l_g^{-1}\circ\bfM$ leads to a nonlinear first-order differential equation.
	Its analysis lies beyond the scope of the present paper.
	
	\subsection[Flat space forms]{Flat space forms}
	\label{sc:flat-space-forms}
	
	The Euclidean and Minkowski cases are simultaneously
	described by the homogeneous principal spin bundles
	\begin{equation*}
		M=\bR^{r,3-r},
		\qquad
		P=\bR^{r,3-r}\rtimes H
		\to\bR^{r,3-r}.
	\end{equation*}
	We choose a global section $s(x)=(x,\mathbf 1).$
	An equivariant map $F\in C^\infty(P,H)^H$
	is uniquely determined by $f(x)=F(x,\mathbf 1)$;
	explicitly,
	\begin{equation*}
		F(x,h)=h^{-1}f(x).
	\end{equation*}
	The flat spin connection gives
	\begin{equation*}
		\scM_{(x,\mathbf 1)}(X)=df_x(X)f(x)^{-1}.
	\end{equation*}
	Consequently, for each fixed admissible non-null spinor
	orbit, the spinor represented by $f$ is a generalized Killing spinor if
	and only if
	\begin{equation}\label{eq:flat-symmetry-condition}
		X\longmapsto
		l_g^{-1}\bigl(df_x(X)f(x)^{-1}\bigr)
	\end{equation}
	is symmetric. By the correspondence theorem in
	\pref{sc:unit-GKS-Lagrangian}, this is also equivalent to
	the associated section $s_f(x)=(x,f(x))$ being $\Omega$-Lagrangian.
	In particular, constant maps $f$ give parallel spinors and Lagrangian
	sections. We do not attempt to classify all symmetric solutions of
	\pref{eq:flat-symmetry-condition} here.
	

	\begin{proposition}\label{pr:GeSp-Lag-flat}
		Let $U\in\frakh$ and $\rho\in C^\infty(\bR)$ be arbitrary and put $X = l_g^{-1}(U)$.
		Then the map
		\begin{equation*}
			f\colon \frakm\ni x
			\longmapsto
			\exp\left(\rho\left(\la x,X\ra\right)U\right)\in H
		\end{equation*}
		satisfies the symmetry condition of \eqref{eq:flat-symmetry-condition}.
		Consequently, it represents a generalized Killing spinor on
		$\bR^{r,3-r}$ and the corresponding graph in
		$\bR^{r,3-r}\rtimes H$ is Lagrangian.
	\end{proposition}
	\begin{proof}
		For $v\in\frakm$, we have
		\begin{equation*}
			df_x(v)f(x)^{-1}
			=
			\rho'\left(\la x,X\ra\right)
			\la v,X\ra U.
		\end{equation*}
		Hence
		\begin{equation*}
			l_g^{-1}\bigl(df_x(v)f(x)^{-1}\bigr)
			=
			\rho'\left(\la x,X\ra\right)
			\la v,X\ra X.
		\end{equation*}
		Therefore, for $v,w\in\frakm$,
		\begin{align*}
			&\la l_g^{-1}\bigl(df_x(v)f(x)^{-1}\bigr),w\ra
			-\la l_g^{-1}\bigl(df_x(w)f(x)^{-1}\bigr),v\ra \\
			&=
			\rho'\left(\la x,X\ra\right)
			\left\{
			\la v,X\ra\la X,w\ra
			-
			\la w,X\ra\la X,v\ra
			\right\}
			=0.
		\end{align*}
		Thus \pref{eq:flat-symmetry-condition} is symmetric.
	\end{proof}
	
	
	\subsection{Generalized Killing spinors on the group-type space forms}
	\label{sc:GKS-group-space-forms}
	Let
	\begin{equation*}
		M=
		\begin{cases}
			\bS^3,&r=0,\\
			\AdS_3,&r=1,
		\end{cases}
		\qquad
		G=H_L\times H_R.
	\end{equation*}
	We identify $M$ with $H$, endowed with its normalized
	bi-invariant metric of constant sectional curvature $ (-1)^r$.
	The principal spin bundle is
	\begin{equation*}
		\pi\colon P=G\longrightarrow M,
		\qquad
		\pi(g_L,g_R)=g_Lg_R^\dagger.
	\end{equation*}
	Its two specializations are
	\begin{alignat}{2}
		\pi\colon\SU(2)_L\times\SU(2)_R
		&\longrightarrow\bS^3,
		&\qquad(g_L,g_R)
		&\longmapsto g_Lg_R^\dagger=g_Lg_R^{-1},
		\label{eq:Spin(4)-to-Spin(3)}\\
		\pi\colon\SU(1,1)_L\times\SU(1,1)_R
		&\longrightarrow\AdS_3,
		&\qquad(g_L,g_R)
		&\longmapsto g_Lg_R^\dagger.
		\label{eq:Spin(2,2)-to-Spin(1,2)}
	\end{alignat}
	The right action of the structure group $H$ is given by
	\begin{equation*}
		(g_L,g_R)\cdot h
		=\bigl(g_Lh,g_R(h^{-1})^\dagger\bigr).
	\end{equation*}
	The corresponding reductive decomposition is
		\begin{equation*}
		\frakh\oplus\frakh
		=
		\{(U,-U^\dagger):U\in\frakh\}
		\oplus
		\{(l_g(X),l_g(X)^\dagger):X\in\frakm\}.
	\end{equation*}
	It follows that the solder form and the spin
	Levi--Civita connection form are
	\begin{equation*}
		\theta
		=l_g^{-1}\frac{\omega_{MC}^L+(\omega_{MC}^R)^\dagger}{2},
		\qquad
		\omegatilde
		=\frac{\omega_{MC}^L-(\omega_{MC}^R)^\dagger}{2}.
	\end{equation*}
	When $r=0, $ we can write simply
	\begin{equation*}
		\theta=l_g^{-1}\frac{\omega_{MC}^L-\omega_{MC}^R}{2},
		\quad
		\omegatilde=\frac{\omega_{MC}^L+\omega_{MC}^R}{2},
	\end{equation*}
	where $\omega^L_{MC}$ (resp. $\omega^R_{MC}$ ) is a Maurer-Cartan form of $H_L$ (resp. $H_R$).
	We use a fixed global section $s(q)=(q,\mathbf 1)$ for $M$ in \pref{sc:GKS-group-space-forms}. Thus
	every equivariant map $F\in C^\infty(P,H)^H$ is determined by
	\begin{equation*}
		f(q)=F(q,\mathbf 1),
		\qquad
		F(g_L,g_R)=g_R^\dagger f(g_Lg_R^\dagger).
	\end{equation*}

	\begin{lemma}\label{lm:F-f-group-space-form}
		Let $\psi_\epsilon \in \Gamma(\Sigma M) $ be a spinor of unit length given by
		\[ \psi_\epsilon = [s,f\mathbb{1}_\epsilon], \]
		for every $f\in C^\infty(M,H). $ Then we have
		\[ 	\mathcal M_{(q,\mathbf1)}(X)
			=
		l_g(X)
		+df_q\bigl(2ql_g(X)\bigr)f(q)^{-1} \]
		for every $X\in\frakm$. 
	\end{lemma}
	
	\begin{proof}
		 For the section $s(q)=(q,\mathbf1),$ since
		\[ ds_q(2ql_g(X)) = (2ql_g(X),0), \] 
		we have
		\begin{equation*}
			(s^*\widetilde\omega)_q
			\bigl(2ql_g(X)\bigr)
			=
			l_g(X).
		\end{equation*}
		Hence the identity follows immediately from
		\pref{lm:F-f-scM}.
	\end{proof}
	
	Let $\epsilon\in\scE_r$, with
	$\scE_r$ as in \pref{sc:unit-GKS-Lagrangian}. The
	subscript $\epsilon$ equals $1$ in Riemannian case and the causal type in Lorentzian case, whereas $\scK^\pm$ describes the choice of sign in the Killing equation. Define
	\begin{align}
		\psi^{\scK^+}_\epsilon(q)
		&=
		[(g_L,g_R),g_R^\dagger\mathbb 1_\epsilon]
		=[(q,\mathbf 1),\mathbb 1_\epsilon],
		\label{eq:positive-Killing-spinor-group-space-form}\\
		\psi^{\scK^-}_\epsilon(q)
		&=
		[(g_L,g_R),g_L^{-1}\mathbb 1_\epsilon]
		=[(q,\mathbf 1),q^{-1}\mathbb 1_\epsilon].
		\label{eq:negative-Killing-spinor-group-space-form}
	\end{align}
	The representing maps $f\colon M \to H$ of the chosen sections are
	\begin{equation*}
		f^{\scK^+}(q)=\mathbf 1,
		\qquad
		f^{\scK^-}(q)=q^{-1}.
	\end{equation*}
	The inverse $g_L^{-1}$ in
	\pref{eq:negative-Killing-spinor-group-space-form} is
	essential: it agrees with $g_L^\dagger$ for $\SU(2)$ but not for
	$\SU(1,1)$. By \pref{lm:F-f-group-space-form},
	\begin{equation*}
		\nabla_\scX\psi^{\scK^\pm}_\epsilon
		=\pm\frac{(\sqrt{-1})^r}{2}\scX\cdot
		\psi^{\scK^\pm}_\epsilon.
	\end{equation*}
	Thus these spinors have associated tensors
	$A=\pm\frac12\id$. For $r=0$, their Killing numbers are
	$\pm\frac12$; for $r=1$, their Killing numbers are
	$\pm\frac{\sqrt{-1}}2$.
	
	For $Y\in\frakm$, define the left- and right-invariant
	vector fields
	\begin{equation*}
		\xi_Y^L(q)=[(q,\mathbf 1),Y],
		\qquad
		\xi_Y^R(q)=[(q,\mathbf 1),\lambda(q^{-1})Y].
	\end{equation*}
	They are divergence-free.
	\begin{proposition}\label{pr:GKS-group-space-forms}
		Let $Y\in\frakm$ be a non-null unit vector. Necessarily
		$||Y||^2=1$ when $r=0$. For $\epsilon\in\scE_r$, set
		\begin{equation*}
			\epsilon_{Y}
			=(-1)^r||Y||^2\epsilon\in\scE_r.
		\end{equation*}
		Then
		\begin{equation*}
			\psi^R_{Y,\epsilon}
			=\xi_Y^R\cdot\psi^{\scK^+}_{\epsilon_{Y}},
			\qquad
			\psi^L_{Y,\epsilon}
			=\xi_Y^L\cdot\psi^{\scK^-}_{\epsilon_{Y}}
		\end{equation*}
		are $\epsilon$-unit generalized Killing spinors.
		In the gauge determined by the global section $s\colon M\ni q\mapsto (q,\mathbf{1})\in P$, their spinor components are
		\begin{align*}
			\psi^R_{Y,\epsilon}(q)
			&=
			[(q,\mathbf 1),
			q^{-1}\gamma(Y)q\mathbb 1_{\epsilon_{Y}}],\\
			\psi^L_{Y,\epsilon}(q)
			&=
			[(q,\mathbf 1),
			\gamma(Y)q^{-1}\mathbb 1_{\epsilon_{Y}}].
		\end{align*}
		Their associated tensors are
		\begin{align}
			A_Y^R(\scX)
			&=
			-\frac32\scX
			+2||Y||^2g(\xi_Y^R,\scX)\xi_Y^R,
			\label{eq:A-right-group-space-form}\\
			A_Y^L(\scX)
			&=
			\frac32\scX
			-2||Y||^2g(\xi_Y^L,\scX)\xi_Y^L.
			\label{eq:A-left-group-space-form}
		\end{align}
	\end{proposition}
	\begin{proof}
		By \pref{lm:F-f-group-space-form}, the same direct calculation as in
		the two special cases gives, for $\iota_{(q,\mathbf 1)}X=\scX|_q$,
		\begin{align*}
			\left.\nabla_{\scX}\psi^{\scK^+}_\epsilon\right|_q
			&=
			\frac{(\sqrt{-1})^r}{2}\scX\cdot
			\psi^{\scK^+}_\epsilon(q),\\
			\left.\nabla_{\scX}\psi^{\scK^-}_\epsilon\right|_q
			&=
			-\frac{(\sqrt{-1})^r}{2}\scX\cdot
			\psi^{\scK^-}_\epsilon(q),\\
			\left.\nabla_{\scX}\psi^R_{Y,\epsilon}\right|_q
			&=
			(\sqrt{-1})^r
			\left(-\frac32\scX
			+2||Y||^2g(\xi_Y^R,\scX)\xi_Y^R\right)
			\cdot\psi^R_{Y,\epsilon}(q),\\
			\left.\nabla_{\scX}\psi^L_{Y,\epsilon}\right|_q
			&=
			(\sqrt{-1})^r
			\left(\frac32\scX
			-2||Y||^2g(\xi_Y^L,\scX)\xi_Y^L\right)
			\cdot\psi^L_{Y,\epsilon}(q).
		\end{align*}
		Hence both associated tensors given by
		\pref{eq:A-right-group-space-form} and
		\ref{eq:A-left-group-space-form} are symmetric.
	\end{proof}
	
	The eigenvalue of $A_Y^R$ on $\bR\xi_Y^R$ is
	$\frac12$, and its eigenvalue on $(\xi_Y^R)^\perp$ is
	$-\frac32$. The eigenvalue of $A_Y^L$ on $\bR\xi_Y^L$ is
	$-\frac12$, and its eigenvalue on $(\xi_Y^L)^\perp$ is
	$\frac32$. Moreover, since $\gamma(\xi)^\dagger(\gamma_0)^r = (-1)^{r+1}(\gamma_0)^r\gamma(\xi), $ then
	\begin{equation}\label{eq:causal-type-after-Clifford}
		\la\xi\cdot\psi,\xi\cdot\psi\ra
		=(-1)^r g(\xi,\xi)\la\psi,\psi\ra.
	\end{equation}
	Thus, in the Lorentzian case, Clifford multiplication
	by a timelike unit vector preserves the causal type of a spinor, whereas
	multiplication by a spacelike unit vector interchanges the spacelike and
	timelike orbits. The subscript $\epsilon$ in
	$\psi^R_{Y,\epsilon}$ and $\psi^L_{Y,\epsilon}$ describes the causal
	type of the resulting spinor, while $\epsilon_{Y}$ describes
	that of the Killing spinor before Clifford multiplication. The superscripts $R$ and $L$ describe which
	invariant vector field, and hence which Killing-number branch, is used.
	For $r=0$, the preceding proposition recovers the
	unit-length generalized Killing spinors on $\bS^3$ described by
	Moroianu and Semmelmann \cite{moroianu2014generalized}. The four
	families are summarized uniformly in
	\pref{tb:GKS-group-space-forms}.
	
	\begin{table}[t]
		\centering
		\small
		\renewcommand{\arraystretch}{1.4}
		\setlength{\tabcolsep}{5pt}
		\begin{tabular}{ccc}
			\hline
			Spinor
			&$A(\scX)$
			&Eigenvalues of $A$
			\\
			\hline
			
			$\psi^{\scK^+}_\epsilon$
			&$\frac12\scX$
			&$\frac12\;(3)$
			\\
			
			$\psi^{\scK^-}_\epsilon$
			&$-\frac12\scX$
			&$-\frac12\;(3)$
			\\
			
			$\psi^R_{Y,\epsilon}$
			&$\displaystyle
			-\frac32\scX
			+2||Y||^2g(\xi_Y^R,\scX)\xi_Y^R$
			&$\frac12\;(1),\ -\frac32\;(2)$
			\\
			
			$\psi^L_{Y,\epsilon}$
			&$\displaystyle
			\frac32\scX
			-2||Y||^2g(\xi_Y^L,\scX)\xi_Y^L$
			&$-\frac12\;(1),\ \frac32\;(2)$
			\\
			
			\hline
		\end{tabular}
		\caption[Generalized Killing spinors on the group-type space forms]
		{Normalized non-null generalized Killing spinors on
			$\bS^3$ and $\AdS_3$. In the Riemannian case, $\epsilon$ takes only
			the value $+1$. The multiplicities of the eigenvalues are
			indicated in parentheses.}
		\label{tb:GKS-group-space-forms}
		\label{tb:GKS-on-S3}
		\label{tb:GKS-on-AdS3}
	\end{table}
	
	The symmetric endomorphisms associated with the four types of generalized Killing spinors in
	\pref{tb:GKS-group-space-forms} are all diagonalizable. These four types do not exhaust the Lorentzian examples arising from extrinsically homogeneous Lagrangian sections on $\AdS_3$.
	
	Starting from the Bianchi type~III family classified by Anarella
	\cite[Example~26]{anarella2026extrinsically}, we construct in \pref{sc:null-Jordan-calculations}
	a discrete family of smooth maps
	\[ f_m\colon\SU(1,1)\longrightarrow\SU(1,1),
	\qquad m\geq2, \]
	whose associated sections are Lagrangian.  
	
	\begin{theorem}
		\label{th:null-Jordan-GKS}
		Let $\ell_m$ be a nowhere-vanishing null vector field \eqref{eq:null-vector-field-m} and let $f_m $ be a $\SU(1,1)$-valued function \eqref{eq:fm} on $\AdS_3$ for each integer $m \geq 2$. 
		Then for every
		$\epsilon\in\scE_1$, the spinor
		\begin{equation*}
			\psi_{m,\epsilon}(q)
			=
			[(q,\mathbf1),f_m(q)\mathbb{1}_\epsilon]
		\end{equation*}
		is an $\epsilon$-unit generalized Killing spinor on
		$\AdS_3$.
		The two causal types have the same symmetric endomorphism $A_m \in \Gamma(\End(T\AdS_3)) $ in \pref{eq:GKS} given by
		\begin{equation}\label{eq:explicit-null-Jordan-tensor}
			A_m
			=
			\frac{1}{2}\id
			-\frac{1}{9}\ell_m\otimes\ell_m^\flat,
			\quad
			\ell_m^\flat=g(\ell_m,\mathord\cdot).
		\end{equation}
		
		In particular, if
		\[
		N_m
		\coloneqq
		A_m-\frac{1}{2}\id,
		\]
		then, for every $q\in\AdS_3$,
		\begin{equation}\label{eq:null-Jordan-properties}
			N_m|_q\neq0,
			\qquad
			\operatorname{rank}N_m|_q=1,
			\qquad
			(N_m|_q)^2=0.
		\end{equation}
		Consequently, $A_m$ is not diagonalizable at any point
		$q\in\AdS_3$.
	\end{theorem}
	
	The construction of $f_m$ and a joint proof of
	\pref{th:null-Jordan-GKS} and
	\pref{pr:null-Jordan-Lagrangian-sections}
	are given in \pref{sc:null-Jordan-calculations}.

	\subsection{Lagrangian sections on the group-type space forms}
	\label{sc:Lagrangian-group-space-forms}
	
	At the fixed global section $ s\colon M\ni q\mapsto (q,\mathbf 1)\in P$, we have
	\begin{equation*}
		X^H|_{(q,\mathbf 1)}
		=\bigl(ql_g(X),l_g(X)^\dagger\bigr),
		\qquad
		U^\#|_{(q,\mathbf 1)}
		=\bigl(qU,-U^\dagger\bigr).
	\end{equation*}
	Let $F\in C^\infty(P,H)^H$ and
	$f(q)=F(q,\mathbf 1)$. The associated section is
	\begin{equation}\label{eq:Lagrangian-section-from-f-group-space-form}
		s_F(q)
		=(q,\mathbf 1)\cdot f(q)
		=\bigl(qf(q),(f(q)^{-1})^\dagger\bigr).
	\end{equation}
	By \pref{lm:F-f-group-space-form},
	\begin{equation*}
		\scM_{(q,\mathbf 1)}(X)
		=l_g(X)+df_q\bigl(2ql_g(X)\bigr)f(q)^{-1}.
	\end{equation*}
	Therefore, for each fixed non-null spinor orbit, the
	following conditions are equivalent:
	\begin{enumerate}
		\item the spinor represented by $f$ is a generalized
		Killing spinor,
		\item $l_g^{-1}\circ\scM_{(q,\mathbf 1)}$ is
		symmetric for every $q\in M$,
		\item the section $s_F$ is Lagrangian with respect to
		$\Omega$.
	\end{enumerate}
	In the Lorentzian case, the same section produces one
	spacelike and one timelike generalized Killing spinor, as stated in
	the correspondence \pref{th:GKS-LagGraph-main} in \pref{sc:unit-GKS-Lagrangian}.
	
	For the two families of Killing spinors defined in
	\pref{eq:positive-Killing-spinor-group-space-form} and
	\pref{eq:negative-Killing-spinor-group-space-form}, the representing $H$-valued functions are
	\begin{equation*}
		f^{\scK^+}(q)=\mathbf 1,
		\qquad
		f^{\scK^-}(q)=q^{-1}.
	\end{equation*}
	The corresponding Lagrangian sections are
	\begin{equation*}
		s_{\scK^+}(q)=(q,\mathbf 1),
		\qquad
		s_{\scK^-}(q)=\bigl(\mathbf 1,q^\dagger\bigr).
	\end{equation*}
	The generalized Killing spinors in
	\pref{pr:GKS-group-space-forms} likewise determine Lagrangian sections
	after the causal orbit of the resulting spinor has been fixed.
	
	For the last two families, let
	$f^R_{Y,\epsilon},f^L_{Y,\epsilon}\colon M\to H$
	be the unique $H$-valued maps determined by
	\begin{align*}
		f^R_{Y,\epsilon}(q)\mathbb 1_\epsilon
		&=
		q^{-1}\gamma(Y)q\mathbb 1_{\epsilon_Y},\\
		f^L_{Y,\epsilon}(q)\mathbb 1_\epsilon
		&=
		\gamma(Y)q^{-1}\mathbb 1_{\epsilon_Y}.
	\end{align*}
	The representing functions are summarized as follows.
	
	\begin{table}[t]
		\centering
		\small
		\renewcommand{\arraystretch}{1.5}
		\setlength{\tabcolsep}{6pt}
		\begin{tabular}{cc}
			\hline
			Spinor
			&$H$-valued representing function
			\\
			\hline
			
			$\psi^{\scK^+}_\epsilon$
			&
			$f^{\scK^+}(q)=\mathbf 1$
			\\
			
			$\psi^{\scK^-}_\epsilon$
			&
			$f^{\scK^-}(q)=q^{-1}$
			\\
			
			$\psi^R_{Y,\epsilon}$
			&
			$\begin{aligned}
				f^R_{Y,\epsilon}(q)\mathbb{1}_\epsilon
				&=q^{-1}\gamma(Y)q\mathbb{1}_{\epsilon_Y}
			\end{aligned}$
			\\
			
			$\psi^L_{Y,\epsilon}$
			&
			$\begin{aligned}
				f^L_{Y,\epsilon}(q)\mathbb{1}_{\epsilon}
				&=
				\gamma(Y)q^{-1}\mathbb{1}_{\epsilon_Y}
			\end{aligned}$
			\\
			
			\hline
		\end{tabular}
		\caption[$H$-valued functions on the group-type space forms]
		{The $H$-valued functions representing the normalized non-null
			generalized Killing spinors on $\bS^3$ and $\AdS_3$.
			In the last two rows, the functions are uniquely determined
			after the non-null spinor orbit has been fixed.}
		\label{tb:H-valued-functions-group-space-forms}
	\end{table}

	\begin{corollary}
		Each $H$-valued function listed in
		\pref{tb:H-valued-functions-group-space-forms}
		determines an $\Omega$-Lagrangian section of $P$.
		In particular, when
		$c=\sqrt{3}$, these sections are Lagrangian
		sections of the nearly pseudo-K\"ahler structure
		$\lb\scG_c,\scJ_c\rb$ on $P$.
	\end{corollary}
	
	\begin{proof}
		Let $f\colon M\to H$ be one of the $H$-valued
		functions listed in
		\pref{tb:H-valued-functions-group-space-forms}.
		By \pref{eq:Lagrangian-section-from-f-group-space-form},
		the corresponding section is
		\begin{equation*}
			s_f(q)
			=
			\lb
			qf(q),
			\lb f(q)^{-1}\rb^\dagger
			\rb.
		\end{equation*}
		By construction, $f$ represents the corresponding
		generalized Killing spinor in
		\pref{tb:GKS-group-space-forms}.
		Hence \pref{th:GKS-LagGraph-main} implies
		\begin{equation*}
			s_f^*\Omega=0.
		\end{equation*}
		Therefore, $s_f$ is an $\Omega$-Lagrangian section.
		
		Finally, when $c=\sqrt{3}$,
		the almost pseudo-Hermitian structure
		$\lb\scG_c,\scJ_c\rb$ is nearly pseudo-K\"ahler
		by \pref{th:NK-str}.
	\end{proof}
	
	When $c=\sqrt{3}, $ we construct a family of extrinsically homogeneous Lagrangian submanifolds of nearly pseudo-K\"ahler structure on $\SU(1,1)_L\times \SU(1,1)_R $ using $\{f_m\in C^\infty(\SU(1,1),\SU(1,1))\}_{m=2}^\infty$ given by \pref{eq:fm} in \pref{sc:null-Jordan-calculations}.
	\begin{proposition}
		\label{pr:null-Jordan-Lagrangian-sections}
		For every integer $m\geq2$, let
		\[
		f_m\colon\SU(1,1)\longrightarrow\SU(1,1)
		\]
		be the smooth map constructed in
		\pref{sc:null-Jordan-calculations}, and let
		\[
		s_m(q)
		=
		\left(
		qf_m(q),
		\bigl(f_m(q)^{-1}\bigr)^\dagger
		\right).
		\]
		Then $s_m$ is an $\Omega$-Lagrangian section of the principal
		spin bundle over $\AdS_3$.
	\end{proposition}
	
	The construction of $f_m$ and a joint proof of
	\pref{th:null-Jordan-GKS} and
	\pref{pr:null-Jordan-Lagrangian-sections}
	are given in \pref{sc:null-Jordan-calculations}.

	
	\section{Lie groups}\label{sc:Lie-group}
	
	In this section, we let $M = G$ be a connected Lie group equipped with a left-invariant pseudo-Riemannian metric. 
	We use the notion of $G$-invariant spin structures and spinors as introduced at the beginning of \pref{sc:constant-curvature spaces} so that lifted $G$-action on $P$ is written by
	\[ \Ltilde_a \colon P\ni (k,h)\mapsto (ak,h) \in P = G\times H, \quad \forall a\in G. \]
	A section $s\in \Gamma(P)$ is called $G$-equivariant if the following equation holds.
	\begin{equation*}
		\Ltilde_a s(q) = s(aq) ,\quad \forall a,q \in G.
	\end{equation*}
	We choose a global section
	\[ s\colon G\ni a \mapsto (a,\mathbf1)\in P=G\times H \] and write left-invariant orthonormal frame as follows
	\[ \scX_k=[s,X_k], \quad k=0,1,2. \] 
	We denote by
	\[
	\tau\colon\frakg\ni \scX\mapsto \tr(\ad_{\scX}) \in \bR
	\]
	the modular form.
	Recall that a Lie group is called unimodular if the linear transformation $\Ad(g) 	$ has determinant $\pm 1$ for every $g\in G.$
	Since $G$ is connected, it is equivalent to $\tau(\scX) = 0$ for every $\scX\in \frakg.$

	\subsection{Invariant generalized Killing spinors}
	\begin{theorem}\label{th:invariant-GKS-unimodular}
		Let $G$ be a connected three-dimensional Lie group with a
		left-invariant pseudo-Riemannian metric and its $G$-invariant spin structure.
		Then for a fixed $\epsilon \in \scE_r, $ the following conditions are equivalent:
		\begin{enumerate}
			\item $G$ is unimodular,
			\item some invariant $\epsilon$-unit spinor is a generalized Killing spinor,
			\item every invariant $\epsilon$-unit spinor is a generalized Killing spinor,
			\item some $G$-equivariant section of $P$ is Lagrangian,
			\item every $G$-equivariant section of $P$ is Lagrangian.
		\end{enumerate}
		Moreover, in this case all invariant generalized Killing spinors have
		the same associated symmetric endomorphism.
	\end{theorem}
	
	\begin{proof}
		Let $\{\scX_k=[s,X_k]\}_{k=0}^2$ be the left-invariant orthonormal frame
		determined by $s$. 
		For every left-invariant vector field $\scX$, the Levi-Civita connection of a left-invariant metric is given by
		\[ \nabla_\scX \scY=\frac12[\scX,\scY]+U(\scX,\scY),\]
		where $U$ is the symmetric bilinear map characterized by
		\[
		2g(U(\scX,\scY),\scZ)
		= g([\scZ,\scX],\scY)+g(\scX,[\scZ,\scY]), \quad \forall \scX,\scY,\scZ \in \frakX^L(G) \simeq \frakg.
		\]
		\begin{align}
			\sdiv \scX_j & = g^{lk}g\lb \frac{1}{2}[\scX_l,\scX_j]+U(\scX_l,\scX_j),\scX_k \rb \notag\\
			& = g^{lk}g\lb [\scX_l,\scX_j],\scX_k\rb\notag \\
			& = -\tr(\ad_{\scX_j}), \quad \forall j\in \{0,1,2\}\label{eq:div-free-frame-Lie group}
		\end{align}
		Since $G$ is connected, it is unimodular if and only if for every $\scX\in \frakg, $
		\[ \tau(\scX) = 0. \]
		This is equivalent to the left-invariant oriented orthonormal frame being divergence-free by \pref{eq:div-free-frame-Lie group}.
		Since the left-invariant trivialization determines a $G$-invariant spin structure $P=G\times H$, then its left action on itself lifts on $P.$
		Then a section of $\pi\colon P\to G$ defined as $s_f = s\cdot f$ in \pref{eq:section-torsor}:
		\begin{equation}\label{eq:s_f-Lie}
			s_f(q) = (q,f(q)) \quad \forall q\in G
		\end{equation} 
		is $G$-equivariant if and only if $f\in C^\infty(G,H) $ is constant.
		By \pref{cr:constant-map-GKS}, for every $h_0\in H, $ as the constant image of $f$, the following are equivalent:
		\begin{enumerate}\renewcommand{\labelenumi}{(\roman{enumi})}
			\item the invariant unit-length spinor represented by $f = s^*F \in C^\infty(G,H) $ is a generalized Killing spinor,
			\item the section $s_f$ given in \pref{eq:s_f-Lie} is Lagrangian,
			\item $\sdiv(\scX_k)=0, $ for all $ k\in \{0,1,2\}$.
		\end{enumerate}
		Moreover, these conditions are independent of $h_0 \in H.$
		Consequently, if $G$ is unimodular, every left-invariant oriented orthonormal frame is divergence-free, and hence every invariant spinor of unit length is a generalized Killing spinor and every $G$-equivariant section is Lagrangian.
		Thus
		\[
		(1)\Longrightarrow(3)\Longrightarrow(2),
		\qquad
		(1)\Longrightarrow(5)\Longrightarrow(4).
		\]
		Conversely, the existence of either an invariant generalized Killing spinor of unit length or an equivariant Lagrangian section implies existence of a divergence-free left-invariant oriented orthonormal frame, 
		and therefore implies $G$ is unimodular.
		
		It remains to prove the final assertion. 
		For the representing $H$-valued function $F\in C^\infty(P,H)^H, $ its expression relative to $s$ is the constant map $ f(q) = F\circ s(q) = h_0. $
		Hence $df = 0, $ then by \pref{eq:scM-gauge} in \pref{lm:scM-def} and \pref{lm:lg-Ad-isom},
		\begin{align*}
			\scM_{s_f(x)}(X_f) & = \Ad_{h_0^{-1}}\scM_{s(x)}(\lambda(h_0)X_f), \quad X_f = s_f^*\theta_x(\scX)\in \frakm\\
			&= \Ad_{h_0^{-1}}\scM_{s(x)}(X), \quad X = s^*\theta_x(\scX)\in \frakm,\\
			[s_f(x),l_g^{-1}\scM_{s_f(x)}(X_f)] & = [s(x)h_0,\lambda(h_0)^{-1}l_g^{-1}\scM_{s(x)}(X)] = [s(x),l_g^{-1}\scM_{s(x)}(X)]\in T_xG,
		\end{align*}
		where $\scX|_x = [s_f(x), X_f] = [s(x),X]\in T_xG.$
		The symmetric endomorphism associated with the corresponding generalized Killing spinor is consequently given by \[ A(\scX) = \frac{1}{2}\ld s, l_g^{-1} (s^*\omegatilde)(\scX)\rd. \]
		This expression depends only on the fixed invariant section $s$ and its induced frame, and is independent of $h_0\in H$ and $\epsilon\in \scE_r.$ Thus all invariant generalized Killing spinors have the same associated symmetric endomorphism.
	\end{proof}
	\begin{example}\label{eg:sol}
		Let $G = \operatorname{Sol} $ be a Lie group of Bianchi type~VI given by
		\[ G\times G \ni \lb (t,x,y), (s,u,v) \rb \longmapsto
		\lb t+s,\ x+e^{-at}u,\ y+e^{at}v\rb \in G. \]
		Assume that $0\neq a \in \bR $ and we put the following left-invariant frame on $G$
		\begin{equation}\label{eq:unimodular-invariant-frame}
			\scT = \partial_t,\quad \scX=e^{-at}\partial_x,\quad \scY = e^{at}\partial_y. 
		\end{equation}
		Put the corresponding global section $\sigma\colon G \ni (t,x,y) \longmapsto ((t,x,y),\mathbf1) \in P$ such that
			\[ \scT = [\sigma,X_0],\quad \scX = [\sigma,X_1],\quad \scY = [\sigma, X_2].\]
		Since nonzero Lie brackets are
		\[ [\scT,\scX] = -a\scX,\quad [\scT,\scY] = a\scY, \] then the modular form is $\tau = 0.$ Hence $G=\operatorname{Sol}$ is unimodular.
		For each left-invariant pseudo-Riemannian metric
		\begin{equation*}
			g_r = (-1)^rdt^2 + e^{2at}dx^2 +e^{-2at}dy^2, 
		\end{equation*}
		\pref{eq:unimodular-invariant-frame} is a divergence-free left-invariant orthonormal frame. 
		By \pref{th:invariant-GKS-unimodular}, every constant map $f\in C^\infty(G,H) $
		defines a $G$-invariant generalized Killing spinor and a corresponding Lagrangian section.
		The corresponding symmetric endomorphisms are obtained in \pref{eq:A-diagonal} by substituting $\nu(t)=at.$
	\end{example}
	
	\subsection{Non-invariant generalized Killing spinors on Lie groups}
	\pref{th:invariant-GKS-unimodular} treats invariant spinors only.
	In this subsection, we obtain some results for generalized Killing spinors that are not left-invariant.
	
	\begin{lemma}
		Let $\{\scX_k=[s,X_k]\}_{k=0}^2 $ be an orthonormal left-invariant frame on $G$ and $0\neq K \in \mathfrak{so}(\frakg,g).$
		For $\vartheta \in C^\infty(G), $ we set
		\[ Q(x) = \exp(\vartheta(x)K),\quad \xi_k = Q\scX_k, \quad \beta_x(\scX) = d\vartheta_x(\scX). \]
		Then $(\xi_0,\xi_1,\xi_2) $ is an orthonormal frame that satisfies 
		\begin{equation*}
			\sdiv\xi_k(x) =(\beta_x \circ K-\tau)(Q(x)\scX_k),\quad \forall x\in G, \quad \forall k\in \{0,1,2\}.
		\end{equation*}
	\end{lemma}
	
	\begin{proof}
		Since $Q=\exp(\vartheta K), $ we have
		\[ \scX_j(Q)\scX_k = d\vartheta(\scX_j)KQ\scX_k, \]
		Hence
		\begin{align*}
			\sdiv(Q\scX_k) &= g^{lj}g(\scX_l, \nabla_{\scX_j}(Q\scX_k))\\
			& = g^{lj}g(\scX_l, \nabla_{\scX_j}(Q^i_k\scX_i))\\
			& = g^{lj}g(\scX_l, \scX_j(Q^i_k)\scX_i + Q^i_k\nabla_{\scX_j}\scX_i)\\
			& = \delta^j_i\scX_j(Q^i_k) + Q^i_kg^{lj}g(\scX_l, \nabla_{\scX_j}\scX_i)\\
			& = d\vartheta(KQ\scX_k) +Q^i_k\sdiv(\scX_i)\\
			& = (\beta\circ K-\tau)(Q\scX_k).
		\end{align*}
	\end{proof}
	\begin{example}\label{eg:Aff}
		Let $G = \Aff^+(\bR)\times \bR $ be a Lie group of Bianchi type~III given by
		\[ G\times G \ni \bigl((t,x,y),(s,u,v)\bigr)
		\longmapsto
		\bigl(t+s,\ x+e^{-at}u,\ y+v\bigr) \in G \]
		for $ a \neq 0. $ We put the following left-invariant frame on $G$
		\begin{equation}\label{eq:non-unimodular-invariant-frame}
			\scT = \partial_t,\quad \scX=e^{-at}\partial_x,\quad \scY = \partial_y. 
		\end{equation}
		The only non-zero Lie bracket is $[\scT,\scX] = -a\scX $, then the modular form is $\tau = -adt.$
		For left-invariant pseudo-Riemannian metric 
		\begin{equation}\label{eq:non-unimodular-metric}
			g_r = (-1)^rdt^2 + e^{2at}dx^2 +dy^2, 
		\end{equation}
		we choose corresponding section $\sigma\colon G \ni (t,x,y) \longmapsto ((t,x,y),\mathbf1) \in P$ such that
		\[ \scT = [\sigma,X_0],\quad \scX = [\sigma,X_1],\quad \scY = [\sigma, X_2], \]
		and define $\vartheta(t,x,y) = ay $ and $ K_r\colon \frakg \to \frakg $ as follows
		\begin{equation*}
			K_r(\scT) = -\scY, \quad K_r(\scX) = 0,\quad K_r(\scY) = (-1)^r\scT.
		\end{equation*}
		\pref{eq:non-unimodular-invariant-frame} is an orthonormal frame for both metrics \eqref{eq:non-unimodular-metric}.
		Since we have $ \beta\circ K_r = \tau,$
		the following non-invariant frames are divergence-free.
		\begin{equation*}
			(\xi_0,\xi_1,\xi_2) = 
			\begin{cases}
				(\cos(ay)\scT-\sin(ay)\scY,\scX,\sin(ay)\scT+\cos(ay)\scY)\quad & (r=0)\\
				(\cosh(ay)\scT-\sinh(ay)\scY,\scX,-\sinh(ay)\scT+\cosh(ay)\scY)\quad & (r=1)
			\end{cases}.
		\end{equation*}
		Let $\widehat K_r\in\frakh$ be the unique element satisfying $\lambda_*(\widehat K_r)=K_r$, i.e.
		\[
		\Khat_r = l_g(X_1) = \frac{(\sqrt{-1})^r}{2}\gamma_1.
		\]
		and define
		\[
		f_r\colon G\ni (t,x,y)\longmapsto \exp\bigl(\vartheta(t,x,y)\widehat K_r\bigr)\in H,
		\]
		i.e.
		\[ f_0(t,x,y) = \begin{pmatrix}
			\cos\frac{ay}{2} & \sin\frac{ay}{2}\\
			-\sin\frac{ay}{2} & \cos\frac{ay}{2}
		\end{pmatrix}, \quad f_1(t,x,y) = \begin{pmatrix}
			\cosh\frac{ay}{2} & \sqrt{-1}\sinh\frac{ay}{2}\\
			-\sqrt{-1}\sinh\frac{ay}{2} & \cosh\frac{ay}{2}
		\end{pmatrix}. \]
		Equivalently, let $F_r\in C^\infty(P,H)^H$ be the
		$H$-equivariant function determined by
		\[
		F_r(\sigma(t,x,y)h)=h^{-1}f_r(t,x,y).
		\]
		For every $\epsilon\in\scE_r$, an $\epsilon$-unit spinor
		\[
		\psi_{r,\epsilon}(t,x,y)
		=
		[\sigma(t,x,y),f_r(t,x,y)\mathbb1_\epsilon]
		\]
		is a generalized Killing spinor.
		By \pref{lm:F-f-scM} and \pref{pr:GKS-condition}, since we have
		\[ df_rf_r^{-1} = al_g(X_1)\otimes\varepsilon^2,\quad \sigma^*\omegatilde = al_g(X_2)\otimes\varepsilon^1, \]
		then the symmetric endomorphism of the generalized Killing spinor is given by
		\[ A = \frac{1}{2}a\lb \scX_1 \otimes \varepsilon^2 + \scX_2\otimes \varepsilon^1 \rb, \]
		where $ \scX_1 = \scX,\ \varepsilon^1 = e^{at}dx $ and $ \scX_2 = \scY, \ \varepsilon^2 = dy. $
		By construction, its characteristic frame is
		$(\xi_0,\xi_1,\xi_2)$.  Since this frame is divergence-free,
		$\psi_{r,\epsilon}$ is a generalized Killing spinor.
		Moreover, by \pref{th:GKS-LagGraph-main}, the corresponding section
		\[
		\sigma_{f_r}(t,x,y)
		=\sigma(t,x,y)f_r(t,x,y)
		=
		\bigl((t,x,y),f_r(t,x,y)\bigr)
		\]
		is an $\Omega$-Lagrangian section of $P$.
		Since $a\ne0$, the map $f_r$ is nonconstant; hence
		$\psi_{r,\epsilon}$ is non-invariant and $\sigma_{f_r}$ is not
		$G$-equivariant.
	\end{example}


	\section{Nonhomogeneous examples}\label{sc:nonhomogeneous}
	In this section, we generalize \pref{eg:sol} by replacing the linear
	function of $t$ appearing in its metric with an arbitrary smooth
	function $\nu(t)$. We thereby obtain examples of unit-length
	generalized Killing spinors and the corresponding Lagrangian sections
	on nonhomogeneous pseudo-Riemannian manifolds.
	
	For $r=0,1, $ let 
	\begin{equation*}
		(M,g_r)=\left(\mathbb R^3,
		(-1)^rdt^2+e^{2\nu(t)}dx^2+e^{-2\nu(t)}dy^2\right),
	\end{equation*}
	be pseudo-Riemannian spin manifolds
	and consider the oriented orthonormal frame and its dual coframe
	\begin{align*}
		\scX_0&=\partial_t,\qquad \scX_1=e^{-\nu(t)}\partial_x,
		\qquad \scX_2=e^{\nu(t)}\partial_y,\\
		\varepsilon^0&=dt,\qquad \varepsilon^1=e^{\nu(t)}dx,
		\qquad \varepsilon^2=e^{-\nu(t)}dy.
	\end{align*}
	The only nonzero brackets and covariant derivatives are
	\begin{align*}
		[\scX_0,\scX_1]&=-\nu' \scX_1,& [\scX_2,\scX_0]&=-\nu' \scX_2,\\
		\nabla_{\scX_1}\scX_0&=\nu' \scX_1,&
		\nabla_{\scX_1}\scX_1&=-(-1)^r\nu' \scX_0,\\
		\nabla_{\scX_2}\scX_0&=-\nu' \scX_2,&
		\nabla_{\scX_2}\scX_2&=(-1)^r\nu' \scX_0.
	\end{align*}
	In particular,
	\begin{equation}\label{eq:diagonal-div-free}
		\sdiv(\scX_0)=\sdiv(\scX_1)
		=\sdiv(\scX_2)=0.
	\end{equation}
	
	Choose a global section $s\colon \bR^3\to P$, so that
	the above basis represented by $\scX_k=[s,X_k]$ for an orthonormal basis $\{X_k\in\frakm\}_{k=0}^2.$
	Then the $H$-equivariant function $F\in C^\infty(P,H)^H$ describing the frame $\{\scX_k\}_{k=0}^2$ is
	\begin{equation*}
		F(s(x)h)=h^{-1}.
	\end{equation*}
	Thus the corresponding unit-length spinor and section are
	\begin{equation*}
		\psi_{0,\epsilon}=[s,\mathbb 1_\epsilon],\quad s_F(q)=s(q).
	\end{equation*}
	Consequently, $s_F$ is the graph of the constant map $\mathbf{1} \in H$ in this
	trivialization. 
	
	For the	 gamma matrices used above, the spin connection is
	\begin{equation*}
		s^*\widetilde\omega
		=\frac{(\sqrt{-1})^r}{2}\nu'\left(\gamma_2\otimes\varepsilon^1
		+\gamma_1\otimes\varepsilon^2\right).
	\end{equation*}
	It follows that
	\begin{equation*}
		\nabla_{\scX_0}\psi_{0,\epsilon}=0,\qquad
		\nabla_{\scX_1}\psi_{0,\epsilon}=\frac{(\sqrt{-1})^r}{2}\nu' \scX_2\cdot\psi_{0,\epsilon},
		\qquad
		\nabla_{\scX_2}\psi_{0,\epsilon}=\frac{(\sqrt{-1})^r}{2}\nu' \scX_1\cdot\psi_{0,\epsilon}.
	\end{equation*}
	Hence $\psi_{0,\epsilon}$ is a generalized Killing spinor with a symmetric endomorphism tensor field
	\begin{equation}\label{eq:A-diagonal}
		A=\frac{\nu'}{2}\left(\scX_2\otimes\varepsilon^1
		+\scX_1\otimes\varepsilon^2\right).
	\end{equation}
	This also follows immediately from \eqref{eq:diagonal-div-free} and the
	divergence-free-frame characterization.
	The tangent space of the corresponding graph is not horizontal unless $\nu'=0$.  At $s(x)$,
	\begin{align*}
		ds(\scX_0)&=X^H_0,\\
		ds(\scX_1)&=X^H_1+\lb \nu'l_g(X_2)\rb^\#,\\
		ds(\scX_2)&=X^H_2+\lb \nu'l_g(X_1)\rb^\#.
	\end{align*}
	Since $F\circ s=\mathbf{1}$, then we have
	\begin{equation*}
		\mathcal M_s(X_i)
		=dF_s(X^H_i)F(s)^{-1}
		=(s^*\widetilde\omega)(\scX_i).
	\end{equation*}
	With the normalization \eqref{eq:lg-gamma},
	we obtain
	\begin{equation*}
		l_g^{-1}\mathcal M_s
		=\nu'\left(X_2\otimes X_1^\flat
		+X_1\otimes X_2^\flat \right).
	\end{equation*}
	This operator is symmetric, and therefore $s_{F}$ is Lagrangian.

	We next construct non-constant examples. For an
	arbitrary smooth function $\vartheta\colon\mathbb R\to\mathbb R$, we put
	\begin{equation*}
		h_\vartheta(t)
		=\exp(2\vartheta(t)l_g(X_0))
		=\exp((\sqrt{-1})^r\vartheta(t)\gamma_0)
		=\begin{pmatrix}
			e^{\sqrt{-1}\vartheta(t)}&0\\
			0&e^{-\sqrt{-1}\vartheta(t)}
		\end{pmatrix}\in H,
	\end{equation*}
	and
	\begin{equation*}
		F_\vartheta(s(q)h)=h^{-1}h_\vartheta(t),
		\quad
		\psi_{\vartheta,\epsilon} = [s,h_\vartheta\mathbb{1}_\epsilon], \quad s_{F_\vartheta} = sh_\vartheta.
	\end{equation*}
	Since
	\[
	dh_\vartheta(\scX_0)h_\vartheta^{-1}
	=2\vartheta'l_g(X_0),
	\qquad
	dh_\vartheta(\scX_1)
	=dh_\vartheta(\scX_2)=0,
	\]
	we obtain
	\begin{align*}
		\scM^\vartheta_{s(p)}(X_0)
		&=2\vartheta'l_g(X_0)
		=(\sqrt{-1})^r\vartheta'\gamma_0,\notag\\
		\scM^\vartheta_{s(p)}(X_1)
		&=\nu'l_g(X_2)
		=\frac{(\sqrt{-1})^r}{2}\nu'\gamma_2,\\
		\scM^\vartheta_{s(p)}(X_2)
		&=\nu'l_g(X_1)
		=\frac{(\sqrt{-1})^r}{2}\nu'\gamma_1,\notag
	\end{align*}
	where, we write $\scM^\vartheta_u(X) = dF_{\vartheta}|_u(X^H)F_\vartheta(u)^{-1} $ for an arbitrary $X\in \frakm.$ 
	Thus every $\psi_{\vartheta,\epsilon}$ is a generalized Killing
	spinor with a symmetric tensor
	\begin{equation*}
		A_\vartheta
		=\vartheta'\scX_0\otimes\varepsilon^0
		+\frac{\nu'}2\left(
		\scX_2\otimes\varepsilon^1+
		\scX_1\otimes\varepsilon^2\right).
	\end{equation*}
	The corresponding section $s_{F_\vartheta}$ given by
	\[ s_{F_\vartheta}(t,x,y) = ((t,x,y), h_\vartheta(t)), \quad \forall(t,x,y) \in M=\bR^3 \]
	is Lagrangian.
	
	Since the characteristic frame of the generalized Killing spinor $\psi_{\vartheta,\epsilon}$ is given by \pref{eq:characteristic-frame-GKS}
	\begin{align}
		\xi_0^\vartheta&=\scX_0,\notag\\
		\xi_1^\vartheta
		&=\cos(2\vartheta)\scX_1+
		\sin(2\vartheta)\scX_2,\label{eq:vartheta-frame}\\
		\xi_2^\vartheta
		&=-\sin(2\vartheta)\scX_1+
		\cos(2\vartheta)\scX_2, \notag
	\end{align}
	then
	\[
	\sdiv(\xi_k^\vartheta)
	=\scX_i\bigl((\lambda(h_\vartheta))^i{}_k\bigr)
	+(\lambda(h_\vartheta))^i{}_k\sdiv(\scX_i)=0.
	\]
	The second term vanishes by \eqref{eq:diagonal-div-free}; in the
	first term, every non-constant coefficient depends only on $t$ and
	appears only in the $\scX_1$- and $\scX_2$-components.
	This provides a second verification of both the generalized Killing
	and Lagrangian properties.
	
	More generally, the characteristic frame associated with the gauge
	$h_\vartheta h_0$, with constant $h_0\in H$, is obtained from
	\eqref{eq:vartheta-frame} by a constant linear combination and is
	therefore still divergence-free.
	In particular, $\vartheta\equiv0$ recovers all constant $H$-valued
	maps in \pref{cr:constant-map-GKS}.
	
	\appendix
	\renewcommand{\thetheorem}{\Alph{section}.\arabic{theorem}}
	\renewcommand{\theequation}{\Alph{section}.\arabic{equation}}
	
	\section{Generalized Killing spinors with non diagonal symmetric endomorphisms on $\AdS_3$}\label{sc:null-Jordan-calculations}
	Throughout this appendix, we assume $c=\sqrt{3}$ and $r=1.$
	The nearly pseudo-K\"ahler structure on $ \SL(2,\bR)\times \SL(2,\bR) $ is studied in \cite{schafer2010nearly}.
	We concern the Bianchi-type~III family of extrinsically homogeneous Lagrangian submanifolds of the nearly pseudo-K\"ahler structure on $\SL(2,\bR)\times \SL(2,\bR)$ in \cite[Example~26]{anarella2026extrinsically} directly in the $\SU(1,1)$ notation of this paper under the following isomorphism
	\begin{equation*}
		\scC\colon \SL(2,\bR)\ni x\mapsto C_0xC_0^{-1}\in \SU(1,1),\quad C_0 = 
		\frac{1}{\sqrt{2}}
		\begin{pmatrix}
			1&\sqrt{-1}\\1&-\sqrt{-1}	
		\end{pmatrix}.
	\end{equation*}
	In this section, we see only one of those family of Lagrangian in  \cite[Example~26]{anarella2026extrinsically} that can be described as a family of global sections
	\[ s_m\colon \SU(1,1)\to \SU(1,1)\times \SU(1,1), \quad 2\le m\in \bN, \]
	except for those Lagrangian sections appeared in \pref{tb:H-valued-functions-group-space-forms}.

	Let $M = \bR\ltimes_{\varphi_1}\bR^2 $ be a Lie group of Bianchi type~III, where
	\[ \varphi_1\colon \bR\ni s\mapsto \begin{pmatrix}
		e^{2s} &0\\ 0&1
	\end{pmatrix} \in \Aut(\bR^2). \]
	The following left-invariant-frame on $M$
	\begin{equation*}
		E_1 = \frac{\partial}{\partial w}, \quad E_2 = e^{2w}\frac{\partial}{\partial u},\quad E_3 = \frac{\partial}{\partial v}
	\end{equation*}
	consists of Bianchi type~III.
	The following $\mathfrak{sl}(2,\bR)$ pseudo-quaternion basis
	\[ \mathbf{i} = \begin{pmatrix}
		1&0\\0&-1
	\end{pmatrix},
	\quad \mathbf{j} = \begin{pmatrix}
		0&1\\1& 0
	\end{pmatrix},
	\quad \mathbf{k} = \begin{pmatrix}
		0&1\\-1&0
	\end{pmatrix} \]
	corresponds to the following $\frakh = \mathfrak{su}(1,1)$ basis
	\[ \scC_*(\mathbf k) = -2l_g(X_0),\quad \scC_*(\mathbf{j}) = 2l_g(X_1),\quad \scC_*(\mathbf i) = -2l_g(X_2).  \]
	We put 
	\begin{equation}\label{eq:null-vector}
		\ell=X_1-X_0,
		\qquad
		\omega_m=\sqrt{\frac{2}{m^2-1}},
		\qquad 2\leq m\in\bN,
	\end{equation}
	and define
	\begin{align}
		Y_m
		&=-\frac{\omega_m^2+11}{6}X_0
		+\frac{\omega_m^2-7}{6}X_1,
		\label{eq:Ym}\\
		Z_m
		&=-\frac{\omega_m^2+9}{6}X_0
		+\frac{\omega_m^2-9}{6}X_1.
		\label{eq:Zm}
	\end{align}
	These vectors satisfy
	\begin{gather}\label{eq:Ym^2-Zm^2}
		\la Y_m,Y_m\ra=-m^2\omega_m^2,
		\qquad
		\la Z_m,Z_m\ra=-\omega_m^2,\\
		Z_m-Y_m=-\frac13\ell,
		\qquad
		\la\ell,Y_m\ra=\la\ell,Z_m\ra=-3.\label{eq:Ym-Zm-relations}
	\end{gather}
	
	For $u,w,v\in\bR$, let
	\begin{align*}
		a(u,w)
		&=\exp\bigl(u l_g(\ell)\bigr)
		\exp\bigl(-2w l_g(X_2)\bigr),\\
		b_m(v)
		&=\exp\bigl(2v l_g(Y_m)\bigr),\\
		k_m(v)
		&=\exp\bigl(2v l_g(Z_m)\bigr).
	\end{align*}
	All three factors belong to $\SU(1,1)$. Under the identification
	$2l_g\colon\frakm\to\frakh$, the Lie algebra of the three-factor
	subgroup is spanned by
	\begin{equation*}
		d\phi_m|_\mathbf{1}(E_1)= 2l_g (-X_2,0,0),
		\quad
		d\phi_m|_\mathbf{1}(E_2) =2l_g \left(\frac12\ell,0,0\right),
		\quad
		d\phi_m|_\mathbf{1}(E_3) = 2l_g (0,Y_m,Z_m)
	\end{equation*}
	via the Lie group homomorphism
	\begin{equation*}
		\phi_m \colon M\ni (w,u,v) \mapsto \lb a(u,w), b_m(v),k_m(v) \rb \in \SU(1,1) \times \SU(1,1) \times \SU(1,1).
	\end{equation*}
	After applying $2l_g$ in each factor, their only nonzero bracket is
	\[
	\left[
	2l_g(-X_2),2l_g\left(\frac12\ell\right)
	\right]
	=
	2l_g(\ell).
	\]
	Thus they span a Bianchi type~III Lie algebra.
	Since for every $X \in \frakm $
	\[ (2l_g(X))^2 = \la X,X\ra I_2, \]
	and \eqref{eq:Ym^2-Zm^2},
	we have
	\begin{align}
		b_m(v)
		&=\cos(m\omega_m v)I_2
		+\frac{\sin(m\omega_m v)}{m\omega_m}\,2l_g(Y_m),\notag \\
		k_m(v)
		&=\cos(\omega_m v)I_2
		+\frac{\sin(\omega_m v)}{\omega_m}\,2l_g(Z_m).\label{eq:k-period}
	\end{align}
	Thus we write the minimal common period of $b_m $ and $k_m$ as $T_m \coloneqq 2\pi/\omega_m. $
	
	\begin{lemma}\label{lm:Iwasawa-null-Jordan}
		The map
		\begin{equation}\label{eq:base-diffeomorphism}
			\bR^2\times \bR/T_m\bZ \ni (w,u,[v])
			\longmapsto a(u,w)k_m(v)^{-1}\in \SU(1,1),
		\end{equation}
		is a diffeomorphism.
	\end{lemma}
	\begin{proof}
		Let 
		\[ AN = \lc \begin{pmatrix}
			r & s \\0&r^{-1} 
		\end{pmatrix} \mid r>0, s\in \bR \rc, \quad D_m = \begin{pmatrix}
			\sqrt{\frac{\omega_m}{3}} & 0\\ 0&\sqrt{\frac{3}{\omega_m}}
		\end{pmatrix}.
		\]
		Since
		\[
		a(u,w) =
		C_0\begin{pmatrix}
			e^w & ue^{-w} \\ 0 & e^{-w}
		\end{pmatrix}
		C_0^{-1}
		,\quad 
		k_m(v) = C_0D_m\begin{pmatrix}
			\cos(\omega_mv) & \sin(\omega_mv) \\ -\sin(\omega_mv) & \cos(\omega_mv)
		\end{pmatrix}D_m^{-1}C_0^{-1},
		\]
		then the standard right-Iwasawa
		decomposition gives a diffeomorphism
		\[
		AN\times\mathcal \SO(2) \ni \lb \begin{pmatrix}
			e^w & ue^{-w} \\ 0 & e^{-w}
		\end{pmatrix}, \begin{pmatrix}
			\cos(\omega_mv) & \sin(\omega_mv) \\ -\sin(\omega_mv) & \cos(\omega_mv)
		\end{pmatrix} \rb
		\longmapsto a(u,w)k_m(v)^{-1}\in \SU(1,1).
		\]
		Consequently,
		$[v]\mapsto k_m(v)$ identifies $\bR/T_m\bZ$ with $k_m(\bR)\simeq \bS^1 $ by
		\pref{eq:k-period}.  This proves the assertion.
	\end{proof}
	
	For every $q\in\SU(1,1)$, let $(u,w,[v])$ be the unique triple such
	that
	\begin{equation*}
		q=a(u,w)k_m(v)^{-1}.
	\end{equation*}
	We define
	\begin{equation}\label{eq:fm}
		f_m(q)=k_m(v)b_m(v)^{-1}.
	\end{equation}
	Equivalently, $f_m$ is the composition of the inverse of
	\pref{eq:base-diffeomorphism} with
	\[
	(u,w,[v])\longmapsto k_m(v)b_m(v)^{-1}.
	\]
	It is also well defined on $[v]$, since
	$T_m$ is a common period of $b_m$ and $k_m$.  Moreover,
	\begin{equation*}
		\bigl(a(u,w)b_m(v)^{-1},k_m(v)b_m(v)^{-1}\bigr)
		=\bigl(qf_m(q),f_m(q)\bigr).
	\end{equation*}
	The corresponding section of the spin bundle is therefore
	\begin{equation*}
		s_m(q)
		=\left(qf_m(q),\bigl(f_m(q)^{-1}\bigr)^\dagger\right).
	\end{equation*}
	Indeed, \pref{eq:Spin(2,2)-to-Spin(1,2)} gives
	$\pi(s_m(q))=q$.
	Let $\ell_m \in \Gamma(T\AdS_3) $ be a vector field given by
	\begin{equation}\label{eq:null-vector-field-m}
		\ell_m|_q = [(q,\mathbf{1}), \lambda(k_m(v))\ell]\in T_q\AdS_3.
	\end{equation}
	This is a nowhere vanishing null vector field on $\AdS_3.$
	
	\begin{proof}[Proof of \pref{th:null-Jordan-GKS} and \pref{pr:null-Jordan-Lagrangian-sections}]
		For every $\epsilon \in \scE_1, $ and $2\le m\in \bN, $ we show that an $\epsilon$-unit spinor $\psi_{m,\epsilon} \in \Gamma(\Sigma \AdS_3)$ given by
		\[ \psi_{m,\epsilon}(q) = [(q,\mathbf{1}), f_m(q)\mathbb{1}_\epsilon] \]
		is a generalized Killing spinor on $\AdS_3.$
		We calculate the associated endomorphism $A_m = \frac{1}{2}l_g^{-1}\circ \bfM $ directly at an arbitrary point.
		We denote 
		\[ a=a(u,w),\quad b=b_m(v), \quad k = k_m(v) \in \SU(1,1). \]
		Every tangent vector to the image of the homomorphism $\phi_m \colon M \to \SU(1,1)\times \SU(1,1)\times\SU(1,1) $ is of the form
		\[ \lb a2l_g(X_a),b2l_g(X_b),k2l_g(X_k) \rb \in T_{(a,b,k)}\SU(1,1)^3, \]
		where $(X_a,X_b,X_k) \in \frakm^3 $ is given by
		\[ (X_a,X_b,X_k) = r\lb -X_2,0,0 \rb + s\lb \frac{1}{2}\ell, 0,0 \rb +t\lb 0,Y_m,Z_m \rb, \quad (r,s,t)\in \bR^3. \]
		Differentiating $q=ak^{-1}$ and $f_m(q)=kb^{-1} $ gives
		\begin{align*}
			q^{-1}\dot{q} & = 2l_g\lb \lambda(k)(X_a-X_k) \rb,\\
			df_m(\dot{q})f_m(q)^{-1} & = 2l_g\lb \lambda(k)(X_k-X_b) \rb.
		\end{align*}
		Since $\lambda(k)(X_a-X_k)$ represents an arbitrary element of $\frak m$,
		the vector
		\[
		\dot q=2ql_g\bigl(\lambda(k)(X_a-X_k)\bigr)
		\]
		represents an arbitrary element of $T_q\AdS_3$.
		By \ref{eq:Ym-Zm-relations}, we have
		\[ \la \ell, X_a-X_k \ra = 3t,\quad X_k-X_b = -\frac{t}{3}\ell = -\frac{1}{9}\la \ell,X_a-X_k \ra\ell. \]
		If we write $X = \lambda(k)(X_a-X_k), $  \pref{lm:F-f-group-space-form} yields
		\begin{align*}
			\frac{1}{2}l_g^{-1}\scM_{(q,\mathbf{1})}(X) & = \frac{1}{2}\lambda(k)(X_a-X_k) + \lambda(k)(X_k-X_b)\\
			& = \frac{1}{2}X -\frac{1}{9}\la \lambda(k)\ell,X \ra\lambda(k)\ell, \quad \forall q \in \AdS_3,\\
			A_m(\scX) & = \frac{1}{2}\scX-\frac{1}{9}g(\ell_m,\scX)\ell_m.
		\end{align*} 			
		In particular, $A_m$ is symmetric. Hence the correspondence between
		generalized Killing spinors and Lagrangian sections implies that $s_m$ is Lagrangian of the nearly pseudo-K\"ahler structure of $\SU(1,1)\times \SU(1,1).$
		Moreover, since $\ell_m$ is nonzero and null, then \pref{eq:explicit-null-Jordan-tensor} and \ref{eq:null-Jordan-properties} hold.
	\end{proof}
	

	\bibliographystyle{amsalpha}
	\bibliography{bibs}

\end{document}